\documentclass[10pt,oneside]{amsart}
\usepackage{amssymb, amscd, amsmath, amsthm, latexsym, hyperref,xcolor}
\usepackage{pb-diagram, fancyhdr, graphicx, psfrag}
\usepackage{tikz}

\newcommand{\compactlist}{\begin{list}{$\bullet$}{\setlength{\leftmargin}{1em}}}
\newcommand{\lmb}{{\em (}}
\newcommand{\rmb}{{\em )}\ }
\newcommand{\KC}{K_C}
\def\W{{ \Lambda}}

\def\ff{{\bf F}}

\def\CP{\C P}
\def\co{\colon\thinspace}
\def\cs{\mathop{\#}}

\def\cala{\mathcal{A}}
\def\calc{\mathcal{C}}

\def\cald{\mathcal{D}}

\DeclareMathOperator{\Int}{Int}
\newcommand{\spinc}{\ifmmode{{\mathfrak s}}\else{${\mathfrak s}$\ }\fi}
\newcommand{\spinct}{\ifmmode{{\mathfrak t}}\else{${\mathfrak t}$\ }\fi}
\newcommand{\spincw}{\ifmmode{{\mathfrak w}}\else{${\mathfrak w}$\ }\fi}
\def\Z{\mathbb Z}
\def\Q{\mathbb Q}
\def\R{\mathbb R}
\def\C{\mathbb C}
\def\newbeta{b}
\DeclareMathOperator{\lcm}{lcm}

 \def\CP{\C P}
 \def\b{\mathrm{b}} 
  \def\t{\textup{t}} 
  \def\ol#1{\overline{#1}}
\newcommand{\Spc}{Spin$^c$}

\def\sm{\setminus}
\def\S{\mathcal{S}}
\def\T{\mathcal{T}}
\def\db{d_\b}
\def\dt{d_\t}
\def\formerdelta{\gamma}
\newcommand{\Spinc}{Spin$^c$}

\newtheorem{thm}{Theorem}

\newtheorem{theorem}{Theorem}[section]
\newtheorem{question}[thm]{Question}
\newtheorem*{theorem*}{Theorem}
\newtheorem{lemma}[theorem]{Lemma}
\newtheorem{corollary}[theorem]{Corollary}
\newtheorem{prop}[theorem]{Proposition}
\newtheorem{proposition}[theorem]{Proposition}
\theoremstyle{definition}
\newtheorem{definition}[theorem]{Definition}
\newtheorem{remark}[equation]{Remark}
\newtheorem{example}[equation]{Example}
\theoremstyle{remark}
\newtheorem*{ack}{Acknowledgements}

\numberwithin{equation}{section}

\begin{document}
\title[Plane curves of arbitrary genus]{Plane algebraic curves of arbitrary genus via Heegaard Floer
homology}

\author{Maciej Borodzik}
\address{Maciej Borodzik: Institute of Mathematics, University of Warsaw, ul. Banacha 2,
02-097 Warsaw, Poland}
\email{mcboro@mimuw.edu.pl}

\author{Matthew Hedden} 
\address{Matthew Hedden: Department of Mathematics, Michigan State University, East Lansing, MI 48824 }
\email{mhedden@math.msu.edu}

\author{Charles Livingston}
\address{Charles Livingston: Department of Mathematics, Indiana University, Bloomington, IN 47405 }
\email{livingst@indiana.edu}

\thanks{The first author was supported by  Polish OPUS grant No 2012/05/B/ST1/03195}
\thanks{The second author was supported by NSF CAREER grant DMS-1150872  and an Alfred P. Sloan Research Fellowship}
\thanks{The third author was supported by
 National Science Foundation   Grant  1007196 and   Simons Foundation Grant 278755}

\begin{abstract} 
Suppose $C$ is a singular curve in $\CP^2$ and it is topologically an embedded surface of genus $g$; such curves are called cuspidal.
The singularities of $C$ are cones on knots $K_i$.
We apply Heegaard Floer theory to find new constraints on the sets of knots $\{K_i\}$ that can arise as the links of singularities of cuspidal curves.  
We combine algebro-geometric constraints with ours to solve the existence problem for curves with genus one,  $d>33$, that possess  exactly one singularity which has exactly one Puiseux pair $(p;q)$.  The realized triples $(p,d,q)$  are expressed as successive even terms in the Fibonacci sequence.
 \end{abstract}

\maketitle

\section{Singular curves}  

Let $C$ be an  algebraic curve   in ${ \CP}^2$, defined as the zero set of a homogeneous polynomial $f$ of degree $d$.  Such a curve is called {\it cuspidal} if the singular points of $C$ are all unibranched; that is, the singular points are isolated and the link of each singularity is a knot in $ S^3$  (such knots are often called {\em algebraic knots}).  Cuspidal curves form a natural family of algebraic curves that are topologically embedded surfaces.

The theory of cuspidal curves of higher genus has not drawn as much attention as the case of rational cuspidal curves, those of topological genus zero. One of the inherent difficulties in the higher genus setting is that the complement of a curve   is not a rational homology ball (in the language of
algebraic geometry, a $\Q$--acyclic surface, see~\cite{FZ1}).  The effect of this is that  one of the main tools in studying rational cuspidal curves, namely the ``semicontinuity
of the spectrum,'' which is a main ingredient of a classification result in~\cite{FLMN06}, 
 becomes considerably less restrictive if the genus is greater than zero.   Section~\ref{ss:semicont} presents a  more detailed discussion of these issues. 

The goal of this paper is to investigate the singularities of cuspidal curves.  To state the main results, we need some background.  To each singular point   one associates the $\delta$--invariant and  the Milnor number $\mu$, which for unibranched singular points are related by $\mu = 2\delta$.  By definition, $\delta $ is the 3--genus of the associated linking circle $K $.    The genus $g$  of $C$ and the $\delta$--invariants, $\{\delta_i\}$, of the set of   singular points,  $\{z_i\}_{i=1}^n$, are related by the {\it genus formula}:  $$g(C) = \frac{(d-1)(d-2)}{2} - \sum_{i=1}^n \delta_i.$$ 

Our goal is to find constraints on the possible sets of singularities beyond those given by the genus formula.  The basic idea is topological.  We suppose that a curve $C\subset \C P^2$ with some collection of singularities exists.  Let $Y(C)$ denote the three-manifold  which arises as the boundary of a closed regular neighborhood $N(C)$ of $C$.  The homeomorphism type of this three-manifold depends only on the degree, genus, and the links of the singularities of $C$.  The complement ${\CP}^2 - \Int(N(C))$ is a smooth four-manifold with particularly simple algebraic topology; in particular, its intersection form is identically zero.    To show that $C$ cannot exist, then, it suffices to show that $Y(C)$ cannot bound a four-manifold with trivial intersection form.    For this, we use invariants derived from the Heegaard Floer homology of $Y(C)$~\cite{os-absolute}.  These invariants are a generalization of the influential ``correction terms" associated to rational homology three-spheres used in the study of homology cobordism groups and knot concordance.  Among the many important references, we mention just one,~\cite{glrs}, which informed our original work on this paper.

 In the present situation, two particularly useful invariants  derived from the Heegaard Floer homology complexes associated to $Y(C)$ present themselves, which  we refer to as the ``bottom" and ``top" correction terms. They depend on a choice of \Spinc\ structure whose  Chern class is  torsion.    The key feature of these invariants is that their values bound the characteristic numbers of smooth negative semi-definite four-manifolds bounded by $Y(C)$, where negative semi-definite means that the self-intersection of any closed surface is non-positive.  Notice that  ${\CP}^2 - \Int(N(C))$ is negative semi-definite with {\em either} of its orientations.   

One must have means to compute the correction terms.  As a first step, we show that $Y(C)$ has a simple description as surgery on  a knot in the connected sum of copies of $S^1 \times S^2$.
With this, along with the fact that the links of the singularities are all $L$--space knots~\cite{Hed09} (in particular, their knot Floer homology complexes are determined by their Alexander polynomials), the computation of the correction terms becomes algorithmic by way of a surgery formula~\cite{os-knotinvariants}.

The statement of our main result uses the notion of the {\it semigroup of a singularity}.  This semigroup of $\Z_{\ge 0}$ is   defined precisely in Section~\ref{subsecsemi}; in the special case that the link of the singularity is a $(p,q)$--torus knot, the semigroup is generated by $p$ and $q$.    If  we are given a finite collection of semigroups $S_1,\ldots,S_n$,  we can define a function $R:\Z_{\ge0}\rightarrow \Z_{\ge0}$ as follows:
\[R(u)=\min_{\substack{k_1+\ldots+k_n=u\\ k_i\ge 0}}\sum_{i=1}^n\#\{S_i\cap[0,k_i)\}.\]
Note that if $k_i=0$ for some $i$, then the number of elements in $\{S_i \cap [0,k_i)\}$ is $0$.
\begin{thm}\label{thm:main}
Suppose $C$ is a cuspidal curve of genus $g$ and degree $d$ 
in $\C P^2$. Let $z_1,\ldots,z_n$ denote its singular points, $S_1,\ldots,S_n$ the corresponding
semigroups, and $R$ the function defined above. Then for any $j=1,\ldots,d-2$ and any $\newbeta =0,\ldots,g$, we have
\begin{equation}\label{eq:main}
0\le  R(jd-2\newbeta+1)-\frac{(j+1)(j+2)}{2}+\newbeta\le g.
\end{equation}
\end{thm}

Unpacking the left inequality in \eqref{eq:main} yields that for any $j=1,\ldots,d-2$ and $\newbeta=0,\ldots,g$, and for any non-negative $k_1,\ldots,k_n$
such that $k_1+\ldots+k_n=jd+1-2\newbeta$ 
\begin{equation}\label{eq:lowerbound}
\sum_{i=1}^n \#\{S_i\cap[0,k_i)\}\ge\frac{(j+1)(j+2)}{2}-\newbeta.
\end{equation}
Stated this way, we find that the case $\newbeta=0$ of  \eqref{eq:lowerbound} is~\cite[Proposition 2]{FLMN06}, which was proved using an elementary dimension counting argument for projective curves.  Indeed, the expression $\sum_{i=1}^n \#\{S_i\cap[0,k_i)\}$ can be interpreted as a  number of linear constraints which is sufficient to ensure that an algebraic curve, viewed as an element in a vector space of homogenous polynomials,  intersects $C$ at $z_i$ with multiplicity $k_i$.  This interpretation leads directly to the left-hand inequality in the case $\newbeta=0$.  It would be interesting to know if algebro-geometric techniques could be used to prove \eqref{eq:lowerbound} for any other values of $\newbeta$ (the argument of~\cite{FLMN06} would need to be altered to incorporate the genus) or, for that matter, if algebraic geometry could shed light on the right-hand inequality in our theorem.  Regardless, it is important to stress that while Theorem~\ref{thm:main} is stated for   algebraic curves, our techniques lie  in the realm of smooth topology; that is, our inequalities are satisfied for  $C^\infty$ maps  $f:C\hookrightarrow \C P^2$  of surfaces which are topological embeddings, and for which there are a finite collection of points $z_i\in C$ satisfying  $df(z_i)=0$ near which $f$ appears holomorphic (that is, within local charts).  It is also worth pointing out that our result can be generalized to surfaces in any smooth 4-manifold with the rational homology of $\C P^2$. In particular, there are direct analogues of Theorem \ref{thm:main} which restrict the cuspidal  curves  in fake projective planes (slight care is needed to account for the image of the inclusion map on the first homology of $C$); for a description of the 50 distinct complex algebraic surfaces with the same Betti numbers as $\C P^2$, 
see~\cite{CartwrightSteger, Mumford,PrasadYeung}).

Theorem~\ref{thm:main} appears to be a useful tool for studying cuspidal  curves  and can effectively obstruct many  configurations of singularities from arising on curves of fixed genus and degree (we give some examples in Section~\ref{examplessect}).  Combining Theorem ~\ref{thm:main} with tools from algebraic geometry yields even better results.  For instance,  we can effectively classify genus one  curves possessing a single singularity of simple form.    From the perspective of algebraic geometry, the theorem  is most naturally stated in terms of {\it Puiseux pairs}; we note that a singularity has one Puiseux pair $(p,q)$ precisely when its link is a $(p,q)$--torus knot (or, equivalently, it is equisingular to $z^p+w^q=0$). Here $p$ and $q$ are positive, coprime integers.

\begin{thm}\label{thm:fib33}
Suppose that $C$ is a cuspidal curve of degree $d>33$,
genus $1$, possessing a single singularity with one Puiseux pair  $(p,q)$. Then  there exists $j>0$ such that $d=\phi_{4j}$ and  $( p,q) =( \phi_{4j-2},\phi_{4j+2})$, where $\phi_0,\phi_1,\ldots$ are the Fibonacci numbers \lmb normalized so $\phi_0=0$, $\phi_1=1$\rmb\!. 
\end{thm}

\noindent In fact, the above is a simplified statement of Theorem~\ref{thm:class}, which additionally provides a finite list  of possible triples $(p,q;d)$ with $d\le 33$. 
The proof of this result uses  Theorem~\ref{thm:main} in conjunction with a {\em multiplicity bound}, expressed in
Theorem~\ref{thm:multbound}.  The latter bounds from above  the degree of a cuspidal curve under consideration   by a linear function of the multiplicity of its singular point (here, the multiplicity is the minimum of $p$ and $q$).  The multiplicity bound, in turn, comes from a general bound on certain numerical invariants of the singular points, the so-called Orevkov $\overline{M}$-numbers.  These numbers are derived from the cohomology of a good resolution of the singular points, and the bound which they satisfy is a consequence of the Bogomolov--Miyaoka--Yau inequality. Note that Theorem~\ref{thm:fib33} is only an obstruction: it says nothing about whether the triples $( \phi_{4j-2},\phi_{4j+2};\phi_{4j})$ are realized by algebraic curves.  As counterpoint, however, we can explicitly construct genus one  curves of degree $\phi_{4j}$ with one cusp and one Puiseux pair using a technique of Orevkov~\cite{Or}:
\begin{thm} \lmb Proposition~\ref{prop:orev} below; cf.~\cite[Theorem C]{Or}\rmb
For any $j=1,2,\ldots$ there exists a curve of genus $1$ and degree $\phi_{4j}$ having a unique singularity with one Puiseux pair $(\phi_{4j-2},\phi_{4j+2})$.
\end{thm}
\noindent One can also produce curves realizing some of the exceptional cases (all of which satisfy   $d\le 33$)  described in Theorem~\ref{thm:class}. Taken together, we solve the geography problem for cuspidal curves with one singularity and one Puiseux pair (modulo a few low degree cases where curves have yet to be constructed).

Computer experiments suggest  that the only instances of Puiseux pairs $(p,q)$ and degrees $d$
that satisfy all the criteria from Theorem~\ref{thm:main}  but fail  the BMY multiplicity bound are those in the family $(p,q) = (a,9a+1)$ and $d = 3a$. It would  be interesting to know whether these are indeed the only additional cases passing the criteria of Theorem~\ref{thm:main}, and whether they can be realized by embedded surfaces in the $C^\infty$ category.  Generalizing this, one can ask:

\begin{question}  Does Theorem~\ref{thm:main}, together with the genus formula, characterize which collections of  algebraic knots can arise as the links of the critical points of a smooth map of a surface $f:C\hookrightarrow \C P^2$ which is a topological embedding? \end{question}  
An affirmative answer would be quite surprising.

\begin{ack}
The authors would like to thank Karoline Moe, Andr\'as N\'emethi and Andr\'as Stipsicz for fruitful discussion.   We recently learned that very similar results to those presented here have been obtained  by Jozef Bodn{\'a}r, Daniele C{\'e}loria, and Marco Golla.
\end{ack}

\section{Overview and Notation}

Let $N(C)$ be a closed regular neighborhood of $C$, having three-manifold  boundary   $Y$.  The complement of the interior of $N(C)$ in $\CP^2$ is a smooth 
four-manifold $X$  with boundary $-Y$.   In the next section we study the algebraic topology of $X$.  In particular, we verify that the intersection form on $H_2(X)$ is identically zero, and study the restriction map $H^2(X)\rightarrow H^2(-Y)$.  Taken together, this information serves as the topological input for the analytic obstructions we consider. 

\subsection{$d$--invariants}  Heegaard Floer homology provides  obstructions to a \Spc\ three-manifold bounding a negative semi-definite \Spc\ four-manifold.  These obstructions are often referred to as $d$--invariants.  To define them, recall that if $\spinc$ is a  \Spc\  structure on $Y$, then Heegaard Floer theory yields a chain complex $CF^\infty(Y, \spinc)$,   freely generated as a module over $\ff[U,U^{-1}]$ (we  use $\ff = \Z_2$ throughout).  The complex is equipped with  a $\Z$ filtration, and the filtered homotopy type of  $CF^\infty(Y, \spinc)$ is an invariant of the pair $(Y,\spinc)$.  In the case that $\spinc$ has torsion first Chern class, the complex has  a  grading by rational numbers. Acting by $U$ in the base ring lowers the filtration level by one and the grading by two. See ~\cite{os-threemanifold} for the definition of  $CF^\infty(Y, \spinc)$ (as a relatively $\Z$-graded complex), and ~\cite{os-absolute} for the definition of its absolute $\Q$-grading.

The complex $CF^\infty(Y, \spinc)$  supports an action by $H_1(Y)$/Torsion which is well-defined up to filtered chain homotopy, and therefore the homology $HF^\infty(Y, \spinc)$ inherits an action by $H_1(Y)$/Torsion (in fact the action on homology extends to the exterior algebra on $H_1(Y)$/Torsion \cite[Section 4.2.5]{os-threemanifold}).  Using this action, we can  define two associated   groups, $HF^\infty(Y, \spinc)_\b$ and $HF^\infty(Y, \spinc)_\t$; the ``b'' and ``t'' are shorthand for ``bottom'' and ``top.''   To define them, one simply considers the kernel and cokernel, respectively, of the $H_1(Y)$/Torsion action.   In the case that $Y$ is a rational homology sphere, the action is zero, so that both groups equal  $HF^\infty(Y, \spinc)$.  In the case that all triple cup products on $H^1(Y)$ vanish (a necessary and sufficient condition that $HF^\infty(Y)$ be ``standard" \cite{lidman})  $HF^\infty(Y, \spinc)_\b$ and  $HF^\infty(Y, \spinc)_\t$ are   isomorphic to $\ff[U,U^{-1}]$.   

There is a  subcomplex  $CF^-(Y,\spinc) = CF^\infty(Y,\spinc)_{  \{i<0\}}$ consisting of elements of filtration level less than 0, and  the corresponding homology group $HF^-(Y,\spinc) $ inherits an $H_1(Y)$/Torsion action.  Thus there are groups $HF^-(Y,\spinc)_\b$ and $HF^-(Y,\spinc)_\t$ and homomorphisms induced by inclusion:  $HF^-(Y,\spinc)_* \to HF^\infty(Y,\spinc)_*$, where $* = \b $ or $ \t $. These top and bottom complexes were first defined in~\cite{os-absolute}.  One useful reference is~\cite{Pts}.  Since then a general theory has been developed in~\cite{levine-ruberman}.  Using these complexes, invariants can be defined as follows.

\begin{definition}  The top and bottom $d$ invariants of the pair $(Y, \spinc)$, denoted $d_\b(Y,\spinc)$ and $d_\t(Y,\spinc)$,   are defined by the property that $(d_*(Y,\spinc) -2)$ is the  maximal grading  among all elements in $HF^-(Y,\spinc)_*$ that map nontrivially into $HF^\infty(Y,\spinc)_*$, where $* = \b $ or $ \t $.
\end{definition}

The analysis of the restriction map $H^2(X)\rightarrow H^2(Y)$ in the next section determines the \Spc\ structures on $Y$ whose $d$-invariants we must compute (to ultimately obstruct the existence of $X$).   To enumerate these \Spc\ structures, we will use the following notation.
\begin{definition}
Suppose $q$ is a positive integer. Define $\S_q$ to be the set of numbers 
\[-(q-1)/2,-(q-1)/2+1,\ldots,(q-1)/2.\] 
\end{definition}
\noindent So, for example, $\S_5=\{-2,-1,0,1,2\}$ and $\S_6=\{-5/2,-3/2,-1/2,1/2,3/2,5/2\}$. 
The following theorem is a restatement of a  Theorem~\ref{dbound2thm}, proved in Section~\ref{sectionbounds}.   It is a consequence of the fact that $X=\C P^2-\mathrm{Int}(N(C))$ is negative semi-definite with either orientation, together with the fact that the $d$-invariants of the boundary of such a four-manifold are bounded by a function determined by its intersection form. 

 \begin{theorem}\label{thm:dbound} Suppose that $C$ is a cuspidal curve in $\C P^2$ with $Y=\partial N(C)$.  Then there is an  enumeration of torsion \Spc\  structures on $Y$, $\{\spinc_m\}$, by integers in the range $\lfloor{\frac{-d^2+1}{2}}\rfloor\le m \le   \lfloor{\frac{ d^2-1}{2}}\rfloor$.  With respect to this enumeration,  for all $k\in\S_d$, the following inequalities are satisfied.
 $$ \db(Y,\spinc_{dk}) \ge -g$$
 and 
 $$ \dt(Y,\spinc_{dk}) \le g.$$ 
 
 \end{theorem}
 
 In order to compute the invariants $d_\b(Y,\spinc)$ and $d_\t(Y, \spinc)$ we need to understand the geometry of $Y$.  Perhaps the most elegant description of $Y$ is as a graph manifold obtained by splicing the circle bundle over the surface of genus $g(C)$ with Euler number $d^2$  to the complements of the links of the singularities of $C$.  For the purposes of computing its Floer homology, however, it is more useful to have a description of $Y$ as obtained by $d^2$ surgery on a knot $\KC$ in $Y_{2g}:=\cs^{2g} S^1\times S^2$. We provide such a description in Theorem~\ref{thm:aboutY}.  Indeed, $\KC$  can be described  as the connected sum $B\#K_1\#K_2\#...\#K_n$, where $B\subset Y_{2g}$ is a simple knot which depends only on the genus of $C$, and $K_i\subset S^3$, $i=1,...,n$ are the  links of the singular points of $C$.

 \subsection{Computing $d$--invariants}  
Let $K$ be a null-homologous knot in a three-manifold $M$ and let $M_k(K)$ denote  the manifold constructed by $k$ surgery on $K$.  For each  \Spc\ structure $\spinc$, the complex  $CF^\infty(M_k(K), \spinc)$ is determined  by a $\Z\oplus\Z$--filtered chain complex $CFK^\infty(M,K,\spinct)$ called the {\em knot Floer homology chain complex}, associated to $K$ and  some \Spc\ structure $\spinct$ on $M$.    In general, the ``surgery formula" relating the knot Floer complex to the complexes of the surgered manifolds can be rather complicated.  For the manifolds arising in this article, however, it will simplify considerably due to the fact that the surgery coefficient  is large with respect to the genus of the knot.  Indeed,  for our purposes it will suffice to understand the homology of subcomplexes of a single doubly filtered chain complex $CFK^\infty(Y_{2g},\KC,\spinc_0)$ associated to $\KC$ and the unique  \Spc\ structure on $\cs^{2g}S^1\times S^2$ having trivial first Chern class.

A key to efficiently understanding this latter complex is that the knots $K_i$ that occur as links of singularities are so-called {\em $L$--space knots}.  For such knots  the complexes $CFK^\infty(S^3, K_i)$ are determined by the Alexander polynomials, $\Delta_{K_i}(t)$.  Moreover,  knot Floer complexes obey a K{\"u}nneth principle under  connected sums:  $CFK^\infty(M\#N, K\#J ) \simeq  CFK^\infty(M, K)\otimes CFK^\infty(N, J)$.  Using this, we have $$  CFK^\infty(Y_{2g} , \KC ) \simeq   CFK^\infty(Y_{2g}, B )\otimes_{i=1}^n CFK^\infty(S^3, K_i ),$$ and   $  CFK^\infty(Y_{2g} , B )$ has been fully described~\cite[Proposition 9.2]{os-knotinvariants}.  Making the connections between these complexes, 
the Alexander polynomials, and the $d$--invariants, leads to the following result.  Details are presented in  Section \ref{section:dfirstpass}.

\begin{theorem}\label{thm:dbotandtop}
There exist invariants $\formerdelta_m$ determined by the Alexander polynomials $\Delta_{K_i}(t)$ with the following property.   If   $\lfloor{\frac{-d^2+1}{2}}\rfloor\le m \le   \lfloor{\frac{ d^2-1}{2}}\rfloor$, then 
 
$$d_{b}(Y,\spinc_m) =\frac{(q-2m)^2-q}{4q}+g-2\max_{\substack{a, b \ge 0\\ a+b = g}}\{\formerdelta_{m+a-b}+ a\}$$ and 
$$d_{t}(Y,\spinc_m) =\frac{(q-2m)^2-q}{4q}+g-2\min_{\substack{a, b \ge 0\\ a+b = g}}\{\formerdelta_{m+a-b} +a\}.$$
 
\end{theorem}
The Alexander polynomial of an algebraic knot can be interpreted in terms of the semigroup of the associated singularity.  Transferring this interpretation to the invariants $\gamma_m$ and combining it with Theorem \ref{thm:dbound} and some algebraic manipulation yields Theorem \ref{thm:main}.

\section{Properties of a  neighborhood of $C$, its boundary, and its complement. }

We continue to let  $N(C)$ denote   a  closed regular neighborhood of $C$.  Let  $Y = \partial {N(C)}$, a closed oriented three-manifold.  The complement of   Int($N(C)$) in $\CP^2$ is a smooth four-manifold $X$  with boundary $-Y$.  In this section we provide a surgery description of $Y$ and homological properties of the pair $(X, -Y)$.

\subsection{A geometric description of $N(C)$ and $Y$.}

To describe ${N(C)}$, we begin with a surface of genus $g$ having a single boundary component.  We denote this surface by $F_g$.  The product $F_g \times D^2$ has boundary $\cs^{2g}S^1 \times S^2$.  Contained in its boundary is the knot $B = \partial F_g \times \{0\}$. Notice that $B$ is null homologous in $\cs^{2g}S^1 \times S^2$.

\begin{theorem} \label{thm:aboutY}  If a cuspidal curve $C$ of degree $d$ has singular points with links $K_i$, then $ {N(C)}$ is built by adding a two handle to $F_g \times D^2$ along the knot $B \cs_{i} K_i$ with framing $d^2$.  In particular, $\partial N(C) = Y$ is built from $\cs^{2g} S^1 \times S^2$ by performing $d^2$ surgery on $B \cs_{i} K_i$.

\end{theorem}

\begin{proof}
The neighborhood  $ {N(C)}$   is constructed in steps as follows.  Let $D_i$ denote a ball neighborhood of the singular point $z_i$.  Joining   $D_1$ to each $D_i$, $i>1$ with a one-handle, each a tubular neighborhood of an arc on $C$, yields a four-ball $D$.   The boundary of $D$ is a three-sphere $S$ with $S \cap C = \#_i K_i$.  The complementary region $C - D$ is diffeomorphic to the surface $F_g$ with neighborhood $D' \cong F_g \times D^2$ having the knot $B$ in its boundary.    Thus, we have $N(C)=  D \cup D'$, with the union identifying a neighborhood of $\#_i K_i$ with a neighborhood of $B$.

The union $D \cup D'$ can be formed in two steps.  First, neighborhoods of a point on $\#_i K_i$ and a point on $B$ are identified.  Since $D$ is a ball, this produces a manifold $D''$ diffeomorphic to $D'$.  The union of the two knots becomes $B \#_i K_i$.  The remainder of the identification is completed by adding a 2--handle to $D''$ along  $B \#_i K_i$.  The framing is $d^2$, that is, the self-intersection of $C$.

\end{proof}

\begin{corollary} $H_1(Y) = \Z_{d^2} \oplus Z^{2g}$ and $H_2(Y) = \Z^{2g}$.\qed
\end{corollary} 

\subsection{The complement $X = {\CP}^2 - \Int(N(C))$} 

The following theorem summarizes  elementary homological calculations.  
 
 \begin{theorem}\label{homologythm}$\ $
 \begin{enumerate} 
 
 \item $H_1(X) \cong \Z_d$ and  $H_2(X) \cong \Z^{2g}$. \vskip.05in
 
 \item The image of the map $\mathrm{Tors}(H^2(X)) \to H^2(Y)$ is isomorphic to $\Z_d \subset \Z_{d^2}$.  \vskip.05in

 \item The map $ H^2(X)/\mathrm{Torsion} \to H^2(Y  )/\mathrm{Torsion}$  is an isomorphism.\vskip.05in
 
 \item {\em Image}$( H^2({\CP}^2) \to H^2(X)) = $ {\em Tors}$(H^2(X))$.     \vskip.05in
  
 \item The intersection form on $H_2(X)$ is identically 0.  \vskip.05in

\end{enumerate}
\end{theorem}

\begin{proof} The map $\Z \cong H_2(C) \to H_2(\CP^2) \cong \Z$ is given by multiplication by $d$.  Using  this, the long exact sequence of the pair $(\CP^2, C)$,  and excision,  yields  $H_1(X,Y) = 0$, $H_2(X,Y) \cong \Z_d \oplus \Z^{2g}$ and $H_3(X,Y) = 0$.

Applying Poincar\'e duality and the universal coefficient theorem yields:  $H_1(X) \cong \Z_d$, $H_2(X) \cong \Z^{2g}$, and $H_3(X) = 0$. In particular, we have part (1) of the theorem.

The long exact sequence of the pair $(X,Y)$   includes the exact sequence  
$$\hskip-1in H_3(X,Y)  \xrightarrow{\partial_3} H_2(Y)\xrightarrow{\iota_2} H_2(X) \xrightarrow{p_2}  H_2(X,Y)\xrightarrow{\partial_2} $$
$$\hskip1in  H_1(Y)\xrightarrow{\iota_1}  H_1(X)\xrightarrow{p_1}  H_1(X,Y)\to  0 $$
which can be written as
$$0 \to \Z^{2g} \xrightarrow{\iota_2} \Z^{2g}  \xrightarrow{p_2} \Z_d \oplus \Z^{2g} \xrightarrow{\partial_2} \Z_{d^2} \oplus \Z^{2g}\xrightarrow{\iota_1}  \Z_d     \to 0 \to  0.$$

We next observe that the map $\partial_2$ must be nonzero on the $\Z_d$ summand. 
 If not, there would be an exact sequence 
 \[\Z^{2g} \to \Z_{d^2} \oplus \Z^{2g} \to \Z_d \to 0.\] 
 Clearly this is impossible:  the image of the initial $\Z^{2g}$ would have to be of rank $2g$. This implies that no element in the image of $\Z^{2g}$ is torsion.
The quotient would then contain elements of order $d^2$.  It immediately follows that the map $p_2$ is the 0 map.

Observe also that $\iota_1$ must be nontrivial on the $\Z_{d^2}$ summand: there is no element in $\Z_d \oplus \Z^{2g}$ that $\partial_2$ could map to an element of order $d^2$.  Given an element of infinite order in $\Z_{d^2} \oplus \Z^{2g}$,  by adding an element from $\Z_{d^2}$ to it  we can assume it is in the kernel of $\iota_1$, and thus in the image of $\partial_2$.   

By duality, the map  $H^2(X) \to H^2(Y)$ corresponds to  the map     $\partial_2 \co \Z_d \oplus \Z^{2g} \to \Z_{d^2} \oplus \Z^{2g}$, which we have now seen is nontrivial on torsion and injective on the free summand. Statements (2) and (3) follow quickly.

To prove (4), we consider a portion of the long exact sequence for the pair $(\CP^2,X)$:  $$H^2(\CP^2) \xrightarrow{\nu_1} H^2(X) \xrightarrow{\nu_2} H^3(\CP^2, X) .$$ We have $H^2(CP^2) \cong \Z$ and    $H^2(X) \cong H_2(X,Y)  \cong \Z_d \oplus \Z^{2g}$.  For the last term we have by excision and Lefschetz duality, $H^3(\CP^2, X) \cong H^3(N(C),Y) \cong H_1(N(C)) \cong H_1(C) \cong \Z^{2g}$.  Thus, our sequence becomes  $$\Z \xrightarrow{\nu_1}  \Z_d \oplus \Z^{2g} \xrightarrow{\nu_2} \Z^{2g} .$$  Clearly $\nu_2$ vanishes on the $\Z_d$ summand, so this summand must be contained in the image of $\nu_1$.  Since the domain of $\nu_1$ is of rank one, the $\Z_d$ summand is precisely the image of $\nu_1$.  The proof of (4) is complete. 

For statement (5), we recall that the intersection form on $H_2(X)$ is given by a composition $H_2(X)\to H_2(X,Y)\to H^2(X)\to Hom(H_2(X),\Z)$.
But the map $H_2(X)\to H_2(X,Y)$ (previously called $p_2$) has already been shown to equal  zero.
\end{proof}

\section{Bounds on the $d$--invariant.}\label{sectionbounds}

Bounds on the $d$--invariants of $Y$ depend on the relationship between \Spc\ structures on $Y$ and those on the complementary space $X$.  We begin with an examination of this relationship and then apply results of~\cite{os-absolute} to attain our desired bounds on the $d$--invariants.

\subsection{\Spc\ structures on $X$ and $Y$.}

\begin{theorem}\label{extendinglemma}  If $C$ is a curve of degree $d$ and $X  = {  \CP^2 - \mathrm{Int}(N(C))}$,  then the torsion \Spc\  structure $\spinc_m$ on $\partial  X$ extends to $X$ if     $m = kd$ for $k\in\S_d$.  Here $\spinc_m$ is the \Spc\  structure on $\partial X$  which extends to a structure $\spinct_m$ on $N(C)$ satisfying  $\left< c_1(\spinct_m), [C]\right> + d^2 = 2m$.

\end{theorem}

\begin{proof} This result is proved in~\cite{bo-li} in the case that $C$ is rational.  Here is an outline of the argument, identifying why it generalizes to the nonrational case.  

There is a \Spc\ structure $\spinct$ on $\CP^2$ having $c_1(\spinct)$ the generator of $H^2(\CP^2)$.  Denote its restriction to $X$ by $\spinct'$.  By Theorem~\ref{homologythm},   $c_1(\spinct')$ is a torsion class in $H^2(X)$ mapping to an element of order $d$ in $H^2(Y)$.  (In the rational case, $H^2(Y)$ is torsion, so the work of Theorem~\ref{homologythm} was not required.)
 
 We have seen that  $H^2(X) = \Z^{2g} \oplus \Z_d$.  Since this cohomology group acts effectively on the set of \Spc\ structures, the orbit of $\spinct'$ under the action of the torsion in $H^2(X)$  is a set of   \Spc\ structures on $X$ with $d$ elements, all that restrict to give torsion \Spc\ structures on $Y$.   The map Torsion($H^2(X)) \to H^2(Y)$ is injective, so these structures are distinct. 
 
 The enumeration of \Spc\ structures as the $\spinc_m$ is described in more detail in~\cite{bo-li}.    
\end{proof}

\subsection{Bounds}  The following  result  provides  bounds on the bottom and top $d$--invariants.

\begin{theorem}\label{dbound2thm} If the complex curve $C$ has degree $d$ and topological genus $g$, then for  $k\in\S_d$,
$$  \db(Y ,\spinc_{dk}) \ge -g$$ and
$$ \dt(Y ,\spinc_{dk}) \le g.$$
\end{theorem}

\begin{proof}This is an application of  ~\cite[Proposition 9.15]{os-absolute}, which says that if $W$ is a negative semi-definite four-manifold for which the restriction map $H^1(W)\rightarrow H^1(\partial W)$ is trivial, then we have the inequality:
$$ c_1(\spinc)^2+ b_2^-(W)\le 4\db(\partial W,\spinc|_{\partial W})+ 2b_1(\partial W),$$ 
where $b_2^-(W)$ is the dimension of the maximal subspace of $H_2(W)$ on which the intersection form is  non-degenerate and $b_1(\partial W)$ is the rank of the first cohomology.  

We apply this proposition to $-X$.   The restriction map $H^1(-X) \to H^1(Y)$ is trivial since $H^1(X) = 0$ and, just as for $X$, the intersection form on $-X$ is zero. Hence $-X$ is negative semi-definite. Now triviality of  the intersection form implies $b_2^-(-X)=0$ and $c_1(\spinc)^2=0$ for any $\spinc\in$\Spc$(-X)$.   Note that $c_1(\spinc)^2$ is defined by lifting a multiple of $c_1(\spinc)\in H^2(-X)$ to $H^2(-X,Y)$ where the intersection form is defined.  Such a lift exists only when $c_1(\spinc|_Y)$ is torsion, but the \Spc\ structures we consider on $Y$ all satisfy this assumption.  Thus the left hand side of the inequality is zero for all $\spinc\in$\Spc$(-X)$.  Since $b_1(Y)=2g$, the inequality becomes:
$$0\le 4\db(Y,\spinc|_{Y})+ 2(2g).$$
This says that $\db(Y,\spinc)\ge -g$ for any \Spc\ structure on $Y$ that extends to $-X$.  But Theorem \ref{extendinglemma} determined exactly which \Spc\ structures on $Y$ extend: they are those of the form $\spinc_{dk}$ where $k\in\S_d$.  This proves the first inequality of the theorem.

To prove the second inequality, we apply the same analysis to the pair $(X,-Y)$, arriving at $$   \db(-Y,\spinc)\ge -g.$$
Now it suffices to show that $\db(-Y,\spinc)=-\dt(Y,\spinc)$.   But this follows easily by observing that the filtration and grading reversing duality isomorphism \cite[Proposition 2.5]{os-threemanifoldapps}:
$$ CF^\infty(-Y,\spinc) \simeq (CF^\infty(Y,\spinc))^*$$ is compatible with the $H_1(Y)$/Torsion action, in the sense that if $\gamma\in H_1$ acts on $CF^\infty(Y)$ by the chain endomorphism $a_\gamma$, then $\gamma$ acts on $CF^\infty(-Y)$ by the adjoint $a_\gamma^*$.  Thus the kernel of the $H_1$ action on $HF^\infty(-Y)$ is identified, by a filtration and grading reversing isomorphism, with the cokernel of the action on $HF^\infty(Y,\spinc)$.  The stated relationship between $\db$ and $\dt$ follows immediately.

\end{proof}

\section{The Heegaard Floer  homology of $Y_{2g} = \#^{2g} S^1 \times S^2 $.}\label{sectionintroduction}

   Given that $Y$ is built as surgery on a knot in $Y_{2g}$, we begin by reviewing the structure of the complex $CF^\infty(Y_{2g})$.  In particular, in this section we describe an explicit basis for this complex  and its homology, and describe the $H_1(Y_{2g})$/Torsion module structure in terms of this basis.  We then describe  the ``top" and ``bottom" Floer homology groups. This description will be used in the next section in conjunction with the knot Floer homology filtration of $K_C$ to     compute the Floer homology of $Y$.   

\subsection{Case of $Y_{1}$}
For $Y_1 = S^1 \times S^2$ and \Spc\ structure $\spinc_0$ with first Chern class $c_1(\spinc_0) = 0$, we have $CF^\infty(Y_1, \spinc_0)\simeq \ff[U,U^{-1}]\oplus \ff[U,U^{-1}]$, where the element $1$  has grading $1/2$ in the first summand and $-1/2$ in the second.     The boundary operator on the complex is trivial, and thus we can identify $CF^\infty(Y_1)$ with $HF^\infty(Y_1)$.    Let $x^*\in H_1(S^1\times S^2)\cong\Z$ be a generator. Then  $x^*$ acts $\ff[U,U^{-1}]$--equivariantly on $HF^\infty(Y_1)$, taking the element $1$ in the first $\ff[U,U^{-1}]$ to the element $1$ in the second.   Thus we can identify 
$$ CF^\infty(Y_1)\simeq \Lambda^*(H^1(Y_1))\otimes \ff[U,U^{-1}],$$
where classes in $H_1(Y_1)$/Torsion act, via the duality pairing between $H_1$/Torsion and $H^1$, on elements in the exterior algebra $\Lambda^*(H^1(Y_1))$.

\subsection{From $Y_1$ to $Y_n$}
There is a K{\"u}nneth principle for the Floer homology of connected sums of three-manifolds \cite[Theorem 6.2]{os-threemanifoldapps},  stating that:
\begin{equation}\label{kunneth}  CF^\infty(M\#N,\spinc_M\#\spinc_N)\simeq CF^\infty(M,\spinc_M)\otimes_{\ff[U,U^{-1}]}CF^\infty(N,\spinc_N).\end{equation}
This homotopy equivalence respects the $\Lambda^*(H_1/\mathrm{Torsion})$ module structure, in the following sense:  there is a natural isomorphism  $H_1(M\#N)\cong H_1(M)\oplus H_1(N)$ 
with which a class  $\gamma\in H_1(M\#N)$ can be identified with  $\gamma_M\oplus\gamma_N\in H_1(M)\oplus H_1(N)$.    Then $\gamma$ acts on $CF^\infty(M\#N)$ as $\gamma_M\otimes Id_N + Id_M\otimes \gamma_N$ under the homotopy equivalence (\ref{kunneth}).

Using this, together with our description of the Floer homology of $Y_1$ above, allows us to conclude that $$ CF^\infty(Y_n)\simeq \Lambda^*(H^1(Y_n))\otimes \ff[U,U^{-1}]$$ as $\Lambda^*(H_1/\mathrm{Torsion})\otimes \ff[U,U^{-1}]$--modules, where $H_1/\mathrm{Torsion}$ classes act by the duality pairing, as above.

\subsection{A useful change of basis for $Y_{2}$}
While the module structure on $CF^\infty(Y_n)$ is completely described above, it will be useful to have an alternate description for $CF^\infty(Y_{2g})$  which will be compatible with the filtration of $CF^\infty(Y_{2g})$ induced by the knot $B$ and, ultimately, $\KC$.  Our description is determined by a change of basis for the  Heegaard Floer homology of  $Y_2 = Y_1\#Y_1$, and the K{\"u}nneth principle above.  Thus we begin with $Y_2$.  Denote the generators of the first cohomology of the two connect summands of $Y_2=Y_1\#Y_1$ by $x$ and $y$.  Thus,  $\Lambda^*H^1(Y_2)$  has basis  $\{1,  x,  y,   x\wedge y \} $.   

 We denote the  hom-dual  generators of $H_1(Y_2)$ as $x^*, y^*$. We have the following alternative description~\cite[Theorem 9.3]{os-knotinvariants} of the action of $H_1(Y_2)$ on the chain complex; recall that the action of $H_1(Y_2)$ commutes with the action of $U$.  We will call the complex equipped with this action the {\em knot adapted complex}. 

\begin{theorem} \label{knotadapted} $CF^\infty(Y_{2}, \spinc_0)  \simeq \Lambda^* H^1(Y_2) \otimes \ff[U, U^{-1}]$ as a module over $\ff[U,U^{-1}]$.  The rational gradings of $1, x, y, $ and $x\wedge y$ are $-1, 0, 0,$ and $1$, respectively.   All these elements are at filtration level $0$. The $\ff[U,U^{-1}]$--equivariant action of $H_1(Y_2)$ on $CF^\infty(Y_{2}, \spinc_0) $ is given by:
\begin{itemize}
\item $ x^*( x\wedge y)  =  y $ 
\item $x^*(  x)  =  1  + U( x\wedge y )$
\item $x^*(   y)     =   0$
\item $x^* (1)   = Uy $.
\end{itemize}
The action of $y^* $ is analogous; see Figure~\ref{fig:actionofx} for a graphical presentation of the action of $x^*$.
\end{theorem}
\begin{proof}  As a graded module over $\ff[U,U^{-1}]$, the above description is clearly isomorphic to our previous description.  To obtain the non-standard (i.e. not induced by the hom-pairing) action of $H_1$/Torsion, we perform the (equivariant, filtered) change of basis  
 $$1\to 1+ Ux\wedge y, \ \ \ \  x\to x,\ \ \ \ y\to y,\ \ \ \ x\wedge y \to x\wedge y$$
 \end{proof}
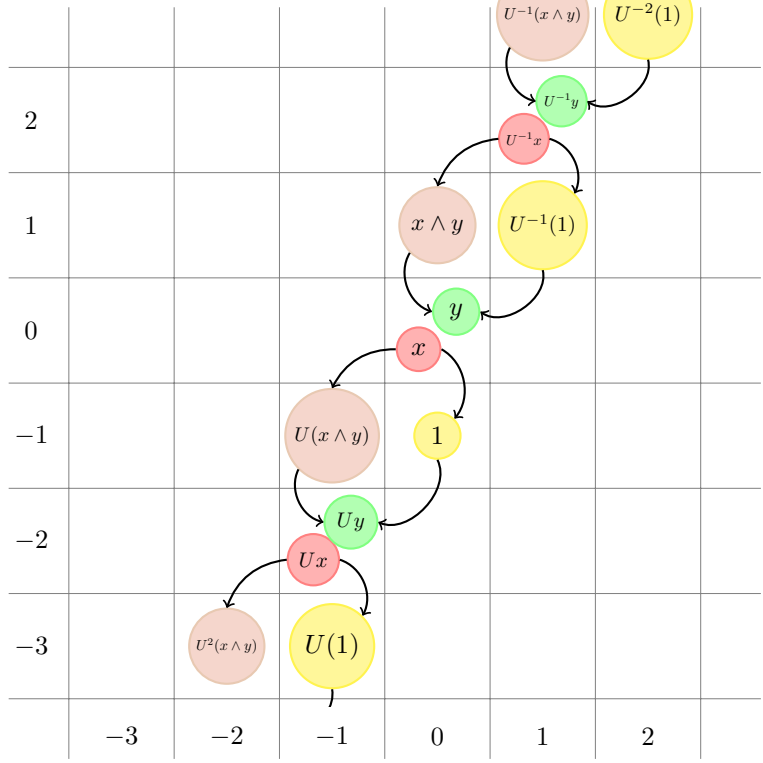
\begin{figure}
\begin{tikzpicture}[first/.style={circle,radius=0.3, draw=yellow!80,fill=yellow!50,thick},xnode/.style={circle,draw=red!50,fill=red!30,thick},%
ynode/.style={circle,draw=green!50,fill=green!30,thick},xynode/.style={circle,draw=green!30!red!30,fill=green!20!red!20,thick}]
\draw[very thin,gray, step=1.4] (-5,-5) grid (5,5);
\foreach \x/\xcor in {-3/-3.5,-2/-2.1,-1/-0.7,0/0.7,1/2.1,2/3.5} { \draw (\xcor,-4.7) node {$\x$}; \draw (-4.7,\xcor) node {$\x$};} 
\begin{scope}
\clip(-4.3,-4.3) rectangle (5.1,5.1);
\node[xnode] (x0) at (0.45,0.45)  {$x$};
\node[ynode] (y0) at (0.95,0.95)  {$y$};
\node[first] (10) at (0.7,-0.7)   {$1$};
\node[xynode, scale=.9] (xy0) at (0.7,2.1)   {$x\wedge y$};
\node[xnode, scale=0.8] (x1) at (-0.95,-2.35) {$Ux $};
\node[ynode, scale=0.8] (y1) at (-0.45,-1.85) {$ Uy$};
\node[first] (11) at (-0.7,-3.5) {$U(1) $};
\node[xynode, scale=0.75] (xy1) at (-0.7,-0.7) {$U(x\wedge y)$};
\node[first, scale=0.8] (1m) at (2.1,2.1) {$ U^{-1} (1)$};
\node[xynode, scale=0.6] (xym) at (2.1,4.9) {$U^{-1} (x\wedge   y)$};
\node[xnode, scale=0.55] (xm) at (1.85,3.25) {$ U^{-1}x$};
\node[ynode, scale=0.55] (ym) at (2.35,3.75) {$ U^{-1}y$};
\node[first,scale=0.8] (1n) at (3.5,4.9) {$U^{-2}(1)$};
\node (xyn) at (3.5,8.4) {};
\node (xn) at (3.25,6.05) { };
\node (yn) at (3.75,6.55) { };
\node (12) at (-2.1,-7.7) {};
\node[xynode, scale=0.55] (xy2) at (-2.1,-3.5) { $U^2(x\wedge y)$};
\node (y2) at (-1.85,-4.65) { };
\node (x2) at (-2.35,-5.15) { };
\draw[->,thick] (xy0.south west) to [bend right=52] (y0.west);
\draw[->,thick] (x0.east) to [bend left=50] (10.north east);
\draw[->,thick] (x0.west) to [bend right=32] (xy1.north);
\draw[->,thick] (10.south) to [bend left=70] (y1.east);
\draw[->,thick] (xy1.south west) to [bend right=52] (y1.west);
\draw[->,thick] (x1.east) to [bend left=50] (11.north east);
\draw[->,thick] (x1.west) to [bend right=32] (xy2.north);
\draw[->,thick] (11.south) to [bend left=70] (y2.east);
\draw[->,thick] (xym.south west) to [bend right=52] (ym.west);
\draw[->,thick] (xm.east) to [bend left=50] (1m.north east);
\draw[->,thick] (xm.west) to [bend right=32] (xy0.north);
\draw[->,thick] (1m.south) to [bend left=70] (y0.east);
\draw[->,thick] (xyn.south west) to [bend right=52] (yn.west);
\draw[->,thick] (xn.east) to [bend left=50] (1n.north east);
\draw[->,thick] (xn.west) to [bend right=32] (xym.north);
\draw[->,thick] (1n.south) to [bend left=70] (ym.east);
\end{scope}
\end{tikzpicture}
\caption{The action of $x^*$ on the $CF^\infty(Y_2,\spinc_0)$. The horizontal coordinate shows the filtration level, the vertical coordinate shows the
grading.}\label{fig:actionofx}
\end{figure}
\begin{remark} Recall that the action of $U$ lowers filtration levels by one and gradings by two.  Thus,   the grading of $U^i (x\wedge y)$ is   $1 - 2i$.  The gradings of $U^i x  $ and $U^i y $ are $-2i$.  The grading of $ U^i (1)$ is $-1 - 2i$.
 \end{remark}

\subsection{Case of $Y_{2g}$}
 Applying the K{\"u}nneth principle to the knot adapted complex gives rise to a model for $CF^\infty(Y_{2g})$ which we will use throughout the article.  In terms of this complex, we will now compute the ``bottom" and ``top" knot Floer homology groups. That is, we will find all  possible chains in the knot adapted complex which are homologous to generators for these groups (both of which are {\em a priori} isomorphic to $\ff[U,U^{-1}]$).  First, we recall the following definition.

 \begin{definition} 
Suppose an abelian group $G$ acts on a second abelian group $H$; that is, there is a homomorphism $G \to \hom(H,H)$.    We define $H_\text{bottom}$ to be the kernel of the action; that is, all elements $h\in H$ such that  $g(h)=0$ for all $g\in G$.   We define $H_{\text{top}}$ to be the cokernel of the action; that is, the quotient of $H$ by the subgroup generated by elements of the form $g(h)$ for some $g\in G$ and $h\in H$. Usually we will abbreviate ``bottom'' and ``top'' by ``b'' and ``t,'' respectively.
 \end{definition} 
We  establish some notation for elements in the complex $CF^\infty(Y_{2g})$:
 \begin{definition}\label{labeling}
Denote the generators of $H^1(Y_{2g}) = H^1(\cs^g Y_2)$   in their  natural order by $\{x_1, y_1, x_2, \ldots , y_g\}$, and let $w_i = x_i \wedge y_i$.  Let $\cala$ denote the set of subsets of $\{1, 2, \ldots, g\}$.  For each $\alpha \in \cala$ we set $w_\alpha = \wedge_{i\in \alpha} w_i$.  For $\alpha \in \cala$ we let $n(\alpha)$ denote the number of elements in $\alpha$. 
 \end{definition}
 
\begin{theorem}\label{lambdathm} 
$CF^\infty(Y_{2g}, \spinc_0)$  is isomorphic to the $\ff[U,U^{-1}]$ module $\W^* H^1(Y_{2g}) \otimes   \ff[U,U^{-1}]$.  The bottom homology is given by 
$$ HF^\infty(Y_{2g},\spinc_0)_\b = \ff[U,U^{-1}] \left[\sum_{\alpha\in \cala} U^{  n(\alpha) -g}   w_\alpha\right].$$
Furthermore, the top homology $  HF^\infty(Y_{2g},\spinc_0)_\t $ is generated by any of the  $U^{ n(\alpha)-g} w_\alpha $,  and any  two such $ U^{ n(\alpha)-g} w_\alpha$ are equivalent in the quotient.  These elements are all of grading level $g$.
\end{theorem}
\noindent  In this statement, the brackets around the summation indicate the homology class represented by the cycle.
\begin{proof} For $Y_2$, the statement is easily verified from our description of the knot adapted complex given in Theorem \ref{knotadapted}.  The general case follows immediately from the K{\"u}nneth principle for connected sums.
\end{proof}

\begin{example} The group $HF_{\b}^\infty(Y_{4},  \spinc_0)$ is generated over $\ff[U, U^{-1}]$ by  $(x_1\wedge y_1 \wedge x_2 \wedge y_2) + U^{-1} (x_1\wedge y_1)+ U^{-1}(x_2\wedge y_2) +U^{-2} (1) $.  The group $HF_{\t}^\infty(Y_{2g}, \spinc_0)$ is generated over $\ff[U, U^{-1}]$ by  either $(x_1\wedge y_1 \wedge x_2 \wedge y_2) , U^{-1}(x_1 \wedge y_1), U^{-1} (x_2 \wedge y_2)$ or $U^{-2}(1)  $, which are equal in the quotient group.
\end{example}


\section{Knot Floer homology and $d$--invariants of surgery}   
 
\label{section:dfirstpass}
\subsection{Description of $CFK^\infty(Y_{2g},B)$}

A null homologous knot in a   three-manifold $M$ induces a second filtration on $CF^\infty(M,\spinc_0)$, called the {\em knot filtration} or {\em Alexander filtration}.  In our case we thus have the doubly filtered complex $CFK^\infty(Y_{2g}, B, \spinc_0)$.  This complex was computed in~\cite{os-knotinvariants} and is described as follows:
\begin{enumerate}
\item As a graded, $\Z$--filtered chain complex $ CFK^\infty(Y_{2g}, B, \spinc_0)\cong CF^\infty(Y_{2g},\spinc_0)$.
\item The knot filtration of an element  $U^i\cdot \zeta$, with $\zeta\in \Lambda^k(H^1(Y_{2g}))$ is given by $-g+k-i$.
\item The $H_1(Y)$/Torsion action is given as in the knot adapted complex.
\end{enumerate}

The following is immediate
\begin{theorem} If $w_\alpha$ is a product of   distinct $w_i$ (according to our labeling convention from Definition \ref{labeling}) then $w_\alpha U^{g- n(\alpha)} \in CFK^\infty(Y_{2g},B, \spinc_0)$ has bifiltration level $(g-n(\alpha), n(\alpha))$, where the value of the second coordinate, $n(\alpha)$, represents the knot filtration level. 
\end{theorem}


\subsection{ Homology of $(Y_{2g},B\# K)$} 

Given a knot $K \subset S^3$, we can form the knot $ B\#K \subset Y_{2g}$.  We will denote this knot by $\KC$, since the case of primary interest will be that arising from a cuspidal curve, whereby $K$ is given as the connected sum of the links of the singular points.  Much of what we say here, however, applies to a general knot in $S^3$.  Like the Heegaard Floer complexes of closed three-manifolds, the knot Floer complexes behave naturally with respect to connected sums, see~\cite[Section 7]{os-knotinvariants}. We have the following.

\begin{theorem} 
$CFK^\infty(Y_{2g}, \KC, \spinc_0)  \simeq CFK^\infty(K) \otimes CFK^\infty(Y_{2g},B,\spinc_0),$ where the bifiltration is additive under tensors. Moreover, under this equivalence, a class $\gamma\in H_1(Y_{2g})$ acts on the knot complex of $\KC$ by $\mathrm{Id}\otimes a_\gamma$ where $a_\gamma$ is the action of $\gamma$ on the complex for $B$.
 \end{theorem}
 
 Recall that $H_*(CFK^\infty(K)) \cong \ff[U,U^{-1}]$.  From the previous theorem along with Theorem~\ref{lambdathm} we have the following. 

\begin{theorem}\label{repthm}  \
\begin{itemize}
\item[(a)] We have $HFK^\infty(Y_{2g},\KC , \spinc_0)_{\b} \cong \ff[U,U^{-1}]$.  Furthermore the generators of grading $g$ are represented by sums  
\[\sum_{\alpha\in\mathcal{A}} a_\alpha \otimes  U^{ n(\alpha)-g} w_\alpha ,\] 
where the $a_\alpha$ are arbitrary cycles of grading $0$ in $CFK^\infty(K)$, 
each representing a generator of $HFK^\infty(K)$.  
\item[(b)] Similarly, $ HFK^\infty(Y_{2g},\KC , \spinc_0)_\t\cong \ff[U,U^{-1}]$, where generators of grading $g$ are 
represented by elements of the form  
\[a_\alpha \otimes  U^{  n(\alpha)-g} w_\alpha.\] 
Here $a_\alpha$ is an arbitrary cycle of grading $0$ in $CFK^\infty(K)$ representing a generator of $HFK^\infty(K)$.
 \end{itemize}
 \end{theorem}
 
\subsection{Computing $d$--invariants of $d^2$--surgery on $(Y_{2g}, \KC)$}
 
We consider $d^2$ surgery on $\KC$ in $Y_{2g}$.  The resulting manifold, $Y_{2g,d^2}(\KC)$, has $H_1(Y_{2g,d^2}(\KC)) \cong \Z^{2g} \oplus \Z_{d^2}$.   There are thus $d^2$ torsion \Spc\ structures on $Y_{2g,d^2}(\KC)$; these come with a natural enumeration by integers $m$,  $\{\spinc_m\}_{-d^2/2 < m \le d^2/2}$, as given 
in~\cite[Section 3.4]{os-knotinvariants} and described below.  We now present a surgery formula  describing the Heegaard Floer homology of these surgered manifolds in terms of the knot Floer complex of $\KC$.

Recall, for a  manifold $M$ with \Spc\ structure $\spinc$ we define 
\[CF^-(M,\spinc) = CF^\infty(M,\spinc)_{\{i<0\}},\] 
the elements of filtration level less than 0. 
The homology  of this complex  is denoted $HF^-(M,\spinc)$.  There is a natural map $ HF^-(M,\spinc) \to HF^\infty(M,\spinc)$.
 
\begin{theorem}[see \expandafter{\cite[Section 4]{os-knotinvariants}}]\label{surgerythm} 
For $d^2 \ge 2g(K) + 2g -1$ and for $-d^2/2 <m < d^2/2$, there is an isomorphism of pairs of $\ff[U]$ modules, 
 $$\hskip-1in(CF^\infty(Y_{2g,d^2}(\KC), \spinc_m), CF^-(Y_{2g,d^2}(\KC),\spinc_m) )\cong $$
 $$\hskip1in( CFK^\infty (Y_{2g},\KC, \spinc_0), CFK^\infty(Y_{2g}, \KC, \spinc_0)_{\{i < 0, j<m\}})[s].$$ 
The grading shift $[s]$ is given by $$   s= \frac{(2m-n)^2 - n}{4n}.$$  If a class is at bi-filtration level $(i,j)$ in  $CFK^\infty (Y_{2g},\KC, \spinc_0)$ then it represents a class at  filtration level $\max(i, j-m)$ in  $CFK^\infty (Y_{2g,d^2}(\KC), \spinc_m)$.
\end{theorem}

\begin{remark} 
In~\cite[Remark 4.3]{os-knotinvariants} the bound given would be presented as $d^2 \ge 2g(\KC) -1$.  We used here the fact that $g(\KC) = g(K)+g$.  Notice that for the knots we are considering, $g(K) + g = \frac{(d-1)(d-2)}{2}$.  Thus, the inequality $d^2 \ge 2g(K) +2g -1$ becomes $d^2 \ge (d-1)(d-2) -1$, which   holds for all $d \ge 1$. 
\end{remark}

Let $\calc \to \cald$ be a map of graded $\ff[U]$ modules.  We denote by  $\formerdelta( \calc, \cald)$   the maximal grading of an element in $\calc$ that maps nontrivially to $\cald$, if defined.
 
Our principal example is the following.   For a manifold $M$ with \Spc\ structure $\spinc$, there is a natural map $ HF^-(M,\spinc) \to HF^\infty(M,\spinc)$.
 
\begin{definition} For $M$ a rational homology sphere and $\spinc$ a  \Spc\ structure, we define $d(M,\spinc) = \formerdelta( HF^-(M,\spinc), HF^\infty(M,\spinc)) +2$.
\end{definition}  
 
 \begin{remark}
The $d$--invariant is often defined in terms of $HF^+(M,\spinc)$.  The equivalence with our the definition is elementary.
\end{remark}
 For general $M$, a similar definition applies to define bottom and top $d$--invariants.
  
 \begin{definition} For general $M$ and $\spinc$ a torsion \Spc\ structure, we define $\db(M,\spinc) = \formerdelta( HF^-(M,\spinc)_\b, HF^\infty(M,\spinc)_\b)+2$ and  $d_\t(M,\spinc) = \formerdelta( HF^-(M,\spinc)_\t, HF^\infty(M,\spinc)_\t)+2$.
 \end{definition}  
 
\noindent  Note that while our definition makes sense for any manifold, it is not as clear what the geometric meaning of $\dt$ and $\db$ are when the three-manifold has non-trivial triple cup products.  
  
  For any knot $K$ for which $CFK^\infty(K)$ is well understood, Theorem~\ref{surgerythm} provides sufficient information to compute $d_\b(Y_{2g, d^2}(\KC))$ and  $ d_\t(Y_{2g, d^2}(\KC))$.  The result is best described in terms of an auxiliary function.
  
 \begin{definition} Let $\T$ be a set of ordered pairs of integers.  For any integer $m$ we define 
$$\formerdelta_m(\T) =\min_{(i,j) \in \T}( \max \{i, j-m\}).$$
  \end{definition}
  
  In brief, $\formerdelta_m$ measures the minimum diagonal distance from an element in $\T$ to the lower left  quadrant with top right vertex $(0,m)$.
 The following result is essentially a corollary of Theorem~\ref{surgerythm}.
 
  \begin{theorem}
  For the complex\;     $ \calc =  CFK^\infty(K)$, let $\T(\calc)$ be the set of all filtrations levels (ordered pairs) of cycles of grading 0 that represent generators of $HFK^\infty(K)$.  For large surgery, $d(S^3_{n}(K), \spinc_m) =-2 \formerdelta_m(\T(\calc)) + ({(2m-n)^2 - n})/{4n}.$
 \end{theorem}
 
\begin{example}\
\begin{itemize}
\item[(a)] If $K$ is a negative trefoil, then $\T=\{(1,1)\}$; if $K$ is a positive trefoil, then $\T=\{(1,0),(0,1)\}$. 
\item[(b)] More generally, suppose that $K$ is a positive $L$--space knot. Then the complex $CFK^\infty(K)$ is what is called a staircase complex, generated by so-called type $A$ elements of grading 0 and type $B$ elements of grading one.  The set $\T$ is the set of all type A vertices of the staircase complex of $K$.  Details are presented, for instance, in~\cite{bo-li}.
\item[(c)] If $K$ is a connected sum of $L$--space knots $K_1,\ldots,K_n$ and $\T_1,\ldots,\T_n$ are the corresponding sets $\T_i=\T(CFK^\infty(K_i))$,
then $\T$ is a set of sums $t_1+\ldots+t_n$, where $t_i\in\T_i$;  see also~\cite{bo-li}.
\end{itemize}
\end{example}
Items (b) and (c) of the above example are the most important in our applications.

 To state the corresponding result for the bottom and top $d$--invariants, we introduce additional notation.  For a set $\T$ of ordered pairs of integers, we let $\T\{a,b\}$ be the same set shifted by $(a,b)$.  Applying Theorem~\ref{repthm} we have the following.
 
 \begin{theorem}\label{dthm} For the complex\; $\calc = CFK^\infty(K)$, let $\T(\calc)$ be the set of all filtrations levels (ordered pairs) of cycles of grading 0 that represent generators of $HFK^\infty(K)$.  Let $n$ be a large integer.
 \begin{enumerate}
 \item $\displaystyle{\db(Y_{2g,n}(\KC),\spinc_m) = -2 \max_{\substack{a, b\ge 0\\ a+b = g}}   \{\formerdelta_m(\T(\calc)\{a,b\})\} +g + s.}$\vskip.05in
 
 \item$\displaystyle{\dt(Y_{2g,n}(\KC),\spinc_m) = -2 \min_{\substack{a, b\ge 0\\ a+b = g}}   \{\formerdelta_m(\T(\calc )\{a,b\}\} +g + s.}$
 
 \end{enumerate}
 where $s = ({(2m-n)^2 - n})/{4n}$
 
 \end{theorem}
 
An elementary calculation restates Theorem~\ref{dthm} in somewhat simpler terms, using  the same notation as in Theorem~\ref{dthm}.
 \begin{theorem}\label{dthm2} For the complex\;     $ \calc =  CFK^\infty(K)$,
 \begin{enumerate}
 \item $\displaystyle{\db(Y_{2g,n}(\KC),\spinc_m) = -2 \max_{\substack{a, b\ge 0\\ a+b = g}}   \{\formerdelta_{m-b+a}(\T(\calc))  +a \}+ g+ s}$\vskip.05in
 
 \item$\displaystyle{\dt(Y_{2g,n}(\KC),\spinc_m) = -2 \min_{\substack{a, b\ge 0\\ a+b = g}}   \{\formerdelta_{m-b+a}(\T(\calc)) + a \}+ g+ s}$
 
 \end{enumerate}
 where $s = ({(2m-n)^2 - n})/{4n}$
 
 \end{theorem}

Combining Theorems~\ref{dbound2thm} and~\ref{dthm2}, we have the following.

\begin{theorem} \label{formfordthm} If $C$ is a cuspidal curve of degree $d$, genus $g$, then for all $k\in\S_d$, and      $ \calc =  CFK^\infty(K)$,
 $$\db( Y_{2g,d^2}(\KC),\spinc_{kd}) = -2 \max_{\substack{a, b\ge 0\\ a+b = g}}   \{\formerdelta_{kd-b+a}(\T(\calc)) + a\}  +g+ s \ge -g$$ and
$$\dt( Y_{2g,d^2}(\KC),\spinc_{kd}) = -2 \min_{\substack{a, b\ge 0\\ a+b = g}}   \{\formerdelta_{kd-b+a}(\T(\calc)) + a\}  +g+ s\le g,$$ where    $s = \frac{(2kd -d^2)^2 -d^2}{4d^2}$ and $K$ is the connected sum of the links of the singularities of $C$.
\end{theorem}

\section{Semigroups, Alexander polynomials, and the $d$--invariant.}

The computation of the obstructions to a set of knots occurring as the links of singularities has been reduced to computing $\formerdelta_m(CFK^\infty(K))$ for particular knots $K$. We will now summarize an interpretation of the value of $\formerdelta_m(K)$ in terms of classical invariants of the singular points and in terms of the Alexander polynomial of $K$.  The material is presented in greater detail  in~\cite{bo-li};  further references include~\cite{HLR12} for a discussion of the Heegaard Floer theory and~\cite{Wall04} for the relationship between semigroups and Alexander polynomials.

\subsection{Semigroup of a singular point}\label{subsecsemi}

 Suppose that $z$ is a cuspidal singular point
of a curve $C$ and $B$ is a sufficiently small ball around $z$. There exists a local parameterization $\psi$ of $C$; that is, a holomorphic map $\psi(t)=(x(t),y(t)) $ mapping a neighborhood of $0 \in \C$ bijectively to a neighborhood of $z \in C$, with $\psi(0) = z$.
For any holomorphic function $F(x,y)$ defined near $z$ we define the order of the zero of $F$ at $z$ to be the order of the zero of  the analytic map $t\mapsto F(x(t),y(t))\in\C$ at 0.
Let $S$ be the set integers  which can be realized as the order for some $F$. 
Then $S$ is clearly a semigroup of $\Z_{\ge 0}$, which we  call  the \emph{semigroup
of the singular point}.  
The \emph{gap sequence}, $G:=\Z_{\ge 0}\sm S$, has precisely $\mu/2$ elements; the largest  is $\mu-1$, where $\mu$ is the Milnor number.

The following two lemmas appear in Lemma 2.4 and a subsequent discussion in~\cite{bo-li}. Further detail can be found in~\cite{Wall04}. 
\begin{lemma}\label{lem:gaplem}  The Alexander polynomial of the link of a singular point can be written as 
$\Delta_K(t)=1+(t-1)\sum_{j=1}^k t^{g_j},$  where  $g_1,\ldots,g_k$ is the gap sequence of the semigroup of the singular point. In particular $k=\#G=\mu/2 = g_k(K)$.
\end{lemma}

If one expands the Alexander polynomial further, the following arises.

\begin{lemma}\label{lem:alex2}
If $K$ is the link of an isolated singularity of a curve $C$ and $\Delta_K(t)$ is expanded as $$\Delta_K(t)=1+(t-1)g(K)+(t-1)^2\sum_{j=0}^{\mu-2}k_jt^{j},$$
then  $k_j=\#\{m>j\colon m\not\in S\}$.
\end{lemma}

\begin{example}\label{ex:t47}
Consider the knot $ T(4,7)$.
Its  Alexander polynomial is
\begin{align*}
\frac{(t^{28}-1)(t-1)}{(t^3-1)(t^7-1)}   =&\    1-t+t^4-t^5+t^7-t^9+t^{11}-t^{13}+t^{14}-t^{17}+t^{18} \\
   =&\ 1+(t-1)(t+t^2+t^3+t^5+t^6+t^{9}+t^{10}+t^{13}+t^{17}) \\=&
\ 1+9(t-1)+\\
+(t-1)^2&\left(9+8t+7t^2+6t^3+6t^4+5t^5+4t^6+4t^7+4t^8+3t^9+\right.\\ &\left.+2t^{10}+2t^{11}+2t^{12}+t^{13}+t^{14}+t^{15}+t^{16}\right).
\end{align*}
The semigroup is $(0,4,7,8,11,12,14,15,16,18,19, 20,21,22,23,\dots)$. The gap sequence is $1,2,3,5,6,9,10,13,17$.
\end{example}

\begin{definition}\label{def:iofg}
For any finite increasing sequence of positive integers $G $, we  define  
\begin{equation}\label{eq:gapfunction}
I_G(m)  =  \#\{ k \in G\cup\Z_{<0} \colon k \ge m\},
\end{equation}
where $\Z_{<0}$ is the set of the negative integers. We  call $I_G$ the \emph{gap function}, because in most applications $G$ will be a gap
sequence, that is the complement  of some semigroup.
\end{definition}
Notice that $k_m=I_G(m+1)$, where $k_m$ is defined in Lemma~\ref{lem:alex2}.

\subsection{Expressing $\formerdelta_m(K)$ in terms of the semigroup}  We now wish to restate Theorem~\ref{formfordthm} in terms of the coefficients of the Alexander polynomial, properly expanded.  For the gap sequence for the knot $K_i$,  denoted $G_{K_i}$, let
$$I_{K_i}(s)=\#\{k\ge s\colon k\in G_{K_i}\cup\Z_{<0}\}. $$

Earlier we defined for a Heegaard Floer complex $\calc = CFK^\infty(K)$ the set of integer pairs $\T(\calc)$ of filtration  levels of cycles in $\calc$ which represent generators of $HFK^\infty(K)$.  By definition we have $\formerdelta_m(\T(\calc) )= \min_{(i,j) \in \T}( \max \{i, j-m\})$.  We have already seen that computing $\formerdelta_m$ is the main step in computing $d$--invariants of manifolds built by surgery on $K$ (or by surgery on $K\cs B\subset Y_{2g}$).  We have the following results.

\begin{theorem} \label{thm:aboutI}{\em(}\cite[Proposition 4.6 ]{bo-li2}{\em)} If the knot $K$ is the link of  a singularity on a cuspidal curve, then $\formerdelta_m(K) = I_{G_{K}}(m +h)$, where $G_K$ is the gap sequence of $K$ and $h$ is its genus.
\end{theorem}

For two functions $I,I'\colon\Z\to\Z$ bounded below we define the following operation:
\begin{equation}\label{eq:diamonddef}  I\diamond I'(s)=\min_{m\in\Z}\{ I(m)+I'(s-m)\}.
\end{equation} 

\begin{theorem}{\em(}\cite[Theorem 5.6]{bo-li2}{\em )} 
For $K = \cs K_i$ with the $K_i$ the links of the singularities on a cuspidal curve, we have $\formerdelta_m(CFK^\infty(K)) = I(m+h)$, where $I =  I_1\diamond\ldots\diamond I_n$, and $h$ is the genus of $K$.
\end{theorem}

\subsection{Proof of Theorem~\ref{thm:main}}

We need some preliminaries.
For a semigroup $S\subset\Z_{\ge 0}$ we introduce another function.
\begin{equation}\label{eq:r-def}
R(m):=\#\left\{ S\cap[0,m) \right\}.
\end{equation}
The function $R$ is closely related to $I(m)$ defined above, in fact in~\cite[Lemma 6.2]{bo-li} it is proved that
\[R(m)=m-h+I(m),\]
where $h=\#(\Z_{\ge 0}\setminus S)$. If $S$ is a semigroup of a (unibranched) singular point, then $h$ is the genus of the link of the singularity.

Given two semigroups $S_1$ and $S_2$, we can consider two gap sequences $G_1$, $G_2$ and the corresponding gap functions $I_1$ and $I_2$. Then 
\[R_1\diamond R_2=m-h+I_1\diamond I_2,\]
where $h=\#(\Z_{\ge 0}\setminus S_1)+\#(\Z_{\ge 0}\setminus S_2)$.

The proof of Theorem~\ref{thm:main} is now a direct application of the above results. By Theorem~\ref{thm:aboutY}, the manifold $Y$
is a surgery on $B \cs K$, where $K$ is a connected sum of the links of singular points of $C$. We use now Theorem~\ref{formfordthm}
together with Theorem~\ref{thm:aboutI} to see that for $k\in\S_d$ we have

\begin{equation}\label{eq:onI}
\begin{split}
\max_{a,b\ge 0,a+b=g} \{I(kd-b+a+h)+a\}-\frac{s}{2}\le& g\\
\min_{a,b\ge 0,a+b=g} \{I(kd-b+a+h)+a\}-\frac{s}{2}\ge& 0,
\end{split}
\end{equation}
where 
\[s=\frac{(2kd-d^2)^2-d^2}{4d^2}=\frac{(d-2k-1)(d-2k+1)}{4}\] 
and $h$ is the genus of the connected sum of links of singularities; that is 
$h=\frac{(d-1)(d-2)}{2}-g=1+d\frac{(d-3)}{2}-g$.

Substituting   $a=g-b$ yields
\[kd-b+a+h=\left(k+\frac{d-3}{2}\right)d-2b+1.\]
We write $j=k+\frac{d-3}{2}$ and notice that $k\in\S_d$ if and only if $j=-1,\ldots,d-2$. Then \eqref{eq:onI} takes the following form
\[0\le I(jd+1-2b)+g-b-\frac{s}{2}\le g,\textrm{ for all $b=0,\ldots,g$}.\]
Expressing $s/2$  in variables $j$ and $d$ yields $\frac{(j-d+1)(j-d+2)}{2}$. Now we replace $I$ by $R$. After straightforward simplifications,  we obtain
\[0\le R(jd+1-2b)+b-\frac{(j+1)(j+2)}{2}\le g.\]
Note that the cases $j=-1,0$ are excluded in the statement of our theorem, as they contain no information.


\section{Examples}\label{examplessect}

We will present here several   applications of   the results of the previous sections, along with  detailed computations.   More substantial applications in algebraic geometry will be presented in  Section~\ref{sec:other} and especially  in Section~\ref{ss:classify}.

\subsection{A degree 21, genus one example.}
Consider the case of $d= 21$.  If a degree $d$ curve is of genus one and has a single singularity of type $(p,q)$, then one would have $$\frac{(d-1)(d-2)}{2} - \frac{(p-1)(q-1)}{2} = 1.$$ This simplifies to $(p-1)(q-1) = 378$.  There are eight relatively prime pairs $(p,q)$ that satisfy this equation:  $(2,379)$, $(3,190)$, $(4,127)$, $(7, 64)$, $(8,55)$, $(10, 43)$, $(15, 28)$, and $(19,22)$. 

For each possibility, Theorem~\ref{thm:main} provides 38 two-sided inequalities that must be satisfied by the associated function $R$.  (The value of $j$ ranges from 1 to $d-2$ and $\newbeta$ ranges from 0 to $g = 1$.)   The first of these inequalities, with $j=1$ and $\newbeta = 0$, is: $$3 \le R(22) \le 4.$$

The semigroup generated by $\{2, 379\}$ contains 11 elements in the interval $[0,22)$, and thus $R(22)=11$ does not satisfy this inequality.  Similarly, the semigroup generated by $\{3,190\}$ contains eight elements in the interval $[0,22)$, and thus $R(22) = 8$ does not satisfy the inequality.  The semigroup  generated by $\{4,127\}$ contains six elements  in the interval $[0,22)$, and thus $R(22) = 6$ does not satisfy the inequality.

In the next two cases, $(7,64)$ and $(8,55)$, all these inequalities are satisfied.  In Section~\ref{sec:other} we will discuss the realization of these curves and place the example $d=21$, $(p,q) = (8,55)$ in a general sequence of realizable curves, related to the fact that $8, 21, $ and $55$ are the  Fibonacci numbers $\phi_6, \phi_8, $ and $\phi_{10}$.

For the pair $(10, 43)$, we need to consider a different value of $j$ to find the first obstruction.   Here we let $j = 2$ and $\newbeta = 0$, giving the inequality
 $$6 \le R(43) \le 7.$$  The semigroup generated by $(10,43)$ contains five elements in the interval $[0,43)$, and thus $R(43)$ does not satisfy the inequality. Finally, we can rule out the possibilities of $(15,28)$ and $(19,22)$ by returning to the inequality $ 3 \le R(22) \le 4.$  In both cases, $R(22) = 2$.
 
\subsection{A degree seven, genus three example}  As a second example, we consider a singular curve of genus 3, showing that there is no degree seven curve of simple type $(4,9)$.   

A generic curve of degree $d=7$ has genus $15$ and the $(4,9)$--torus knot has genus 12.  Thus a degree seven curve of simple type $(4,9)$ would have genus 3.  Theorem~\ref{thm:main} provides 20 two-sided inequalities; the value of $j$ is between 1 and 5 and the value of $\newbeta$ is between 0 and 3.  Of these inequalities, exactly two provide obstructions. 
 For $j=1$, $\newbeta = 0$ and $j=3$, $\newbeta = 3$, we have the constraints:
$$ 3 \le R(8) \le 6 ,$$

$$ 7 \le R(16) \le10.$$

The semigroup generated by $\{4,9\}$ has two element in $[0,8)$, so $R(8) = 2$ does not satisfy the first inequality.  This semigroup contains six elements in the interval $[0,16)$  (these elements are $\{0,4,8,9,12,13\}$) and thus $R(16) = 6$ does not satisfy the second inequality.
 
\subsection{A degree nine, genus eight example}  The obstructions given by Theorem~\ref{thm:main} become weaker as the genus increases, necessarily so, since more singularity types can be realized.  We present here one more example, one in which the obstruction remains effective despite the genus being large relative to $d$.  We consider the case of $d=9$ and the curve being of simple type $(5,11)$.  

Since the generic genus of a degree nine curve is $28$ and the $(5,11)$--torus knot has genus $20$, a simple curve of degree nine and  of simple type $(5,11)$ would have genus eight.  Thus, Theorem~\ref{thm:main} provides $63$ two-sided inequalities, as $j$ ranges from $1$ to $7$ and $\newbeta$ ranges from $0$ to $8$.  Precisely one of these provides an obstruction.  In the case $j=5, \newbeta =8$ we get inequalities
$$ 13 \le R(30) \le21.$$
The semigroup generated by $\{5, 11\}$ contains  12 elements in the interval $[0,30)$, and thus the  inequality is not satisfied.

\subsection{A singularity $T(4,7)$ on a degree six curve}\label{ss:t47}

The singularity was discussed briefly in Example~\ref{ex:t47}.  Since a generic degree six curve has genus 10 and the $(4,7)$--torus knot has genus nine, a degree six curve of simple type $(4,7)$ is of genus one.  There are eight constraints given by Theorem~\ref{thm:main}.  Two of these are 
$$ 3 \le R(7) \le 4, \text{ and}$$
$$ 5 \le R(11) \le 6.$$
Since for $(4,7)$, $R(7) = 2$ and $R(11) = 4$, these inequalities are violated.  

This example is of special interest.   We will see in Example~\ref{ex:t47-spec} that another important criterion, {\it  semicontinuity of the spectrum}, is insufficient
to obstruct this case.

\section{Genus one curves with one simple singularity}\label{sec:other}
In the section we prove our classification result for genus one curves with a single singular point with one Puiseux pair.   The bulk of the work lies in the obstruction of curves, for which we use Theorem~\ref{thm:main} together with the multiplicity bound from Section~\ref{sec:multbound} below.  For the sake of exposition, we introduce some (non-standard) terminology. Throughout the section, we will call a singularity {\it simple} of type $(p,q)$ if its link is a $(p,q)$ torus knot; that is, if it has a single Puiseux pair, $(p,q)$.  We will similarly say a curve $C$ is of \emph{simple type $(p,q)$}, if it has precisely one singularity and that singularity is of simple type $(p,q)$.  Let $\phi_0,\phi_1,\ldots$ be the sequence of Fibonacci numbers such that $\phi_0=0$, $\phi_1=1$ and $\phi_{n+1}=\phi_n+\phi_{n-1}$.  The main theorem of this section is the following.

\begin{theorem}\label{thm:class}
Suppose $C\subset \C P^2$ is an algebraic curve of genus one,  degree $d$,  
and of simple type $(p,q)$. Then either: (A) $d=\phi_{4n}$, $p=\phi_{4n-2}$, $q=\phi_{4n+2}$ for some $n>0$;  or (B) the values of  $(p,q)$ and $d$ are on  the following list.

\begin{itemize}
\item[(a)] $(p,q) = (2,5), d =4${\rm ;}
\item[(b)] $(p,q) = (2,11), d = 5${\rm ;}
\item[(c)] $(p,q) = (3,10), d= 6${\rm ;}
\item[(d)] $(p,q) = (6,37), d= 15${\rm ;}
\item[(e)] $(p,q) = (9,64) , d= 24${\rm ;}
\item[(f)] $(p,q) = (10,73), d= 27${\rm ;}
\item[(g)] $(p,q) = (12,91), d=  33${\rm ;}
\item[(h)] $(p,q) = (p,9p+1), d =  3p$ for $p=2,\ldots,10$.
\end{itemize}
\end{theorem}

\begin{remark}  {$\hskip.1in \ \  $}
\begin{itemize}
\item[(a)] Theorem~\ref{thm:class} does not state  that any of these cases can be   realized as an algebraic curve, nor does it state in how many ways each case can be realized if some realization exists.   In Proposition~\ref{prop:orev} and Proposition~\ref{prop:moe} we clarify that  cases (a)--(c)
can be  realized by an algebraic curve and that the main case $(\phi_{4n-2},\phi_{4n+2})$ can be realized.
\item[(b)] All the special cases have degree at most $33$. 
\end{itemize}\end{remark}

We begin with the following simple result. 

\begin{prop}If a degree $d$ curve is of genus one and has one simple singularity of  type $(p,q)$, then $(p-1)(q-1) = d(d-3)$.
\end{prop}

\begin{proof}  This is an immediate corollary of the genus formula, restating the condition that  $  (d-1)(d-2)/2  = (p-1)(q-1)/2    +1$.\end{proof}

The rest of this section is devoted to proving Theorem~\ref{thm:class}. 

\subsection{Preliminary bounds}  Assuming that $C$ satisfies the assumptions of Theorem~\ref{thm:class} and the $d$, $p$, and $q$ are as in the statement of that theorem, we begin by developing some basic bounds.

\begin{lemma} If $d\ge 5$, then  $p\le d-3$ and $q\ge d+3$.
\end{lemma}

\begin{proof}
First observe that $p \le d-1$: since $p<q$ and $(p-1)(q-1) = (d-3)d$, it  immediately follows that  $p-1 \le d-2$.  

We now improve this to show that $p  \le d-2$. By Theorem~\ref{thm:main}, setting $j=1$ and $\newbeta = 1$, we find $$2 \le R(d-1) \le 3.$$

If $p = d-1$, then 0 is the only element of the semigroup generated by $p$ and $q$ that is in the interval $[0,d-1)$, in which case $R(d-1) =1$, giving a contradiction.  

Finally, we consider the case that  $p = d-2$, which we write as   $p-1 = d-3$.  Clearly $q-1 = d$ and $q = d+1$.
  By Theorem~\ref{thm:main}, setting $j=2$ and $\newbeta = 1$ we find 
$$5 \le R(2d-1) \le 6.$$
The following  integers are the first six  elements in the semigroup generated by $d-2 $ and $d+1$ in increasing order:  $$\{0, d-2, d+1, 2(d-2), (d-2)+(d+1), 2(d+1)\} =$$ $$\{0, d-2, d+1, 2d-4, 2d-1, 2d+2\}.$$
Thus, $R(2d-1) \le 4$, with equality whenever $d>5$, giving the desired contradiction.  (If $d = 4$, then the element $3d -6$ would also be an element in the semigroup that is less than $2d-1$.)

 For the  lower bound  on $q$, we observe that the minimum value of $q$ would occur if   $p = d-3$.  Solving  for $q$ yields   $q =d+2 +\frac{4}{d-4}$. Since $q$ is an integer and $d >4$, it follows that $q \ge d+3$.   
 \end{proof}
 
 We now place a stronger upper bound on $p$ and a lower bound on $q$.
 
\begin{lemma}\label{lemma-d/2}Suppose $C$  satisfies the conditions  of Theorem~\ref{thm:class}.  If $d > 6$, then $p < \frac{1}{2}d$ and $q \ge 2d-1$.  
\end{lemma}

\begin{proof}
First observe   that by Theorem~\ref{thm:main} with $j=1$ and $\newbeta=0$, $$3\le R(d+1) \le 4.$$

If $p> \frac{1}{2}d$ then there are at most two elements (0 and $p$) in the semigroup generated by $p$ and $q$ in the interval $[0,d+1)$, giving a contradiction.

If $p= \frac{1}{2}d$, then one computes   $q-1=\frac{2(d-1)(d-2)-2}{d-2}$, which is not an integer   since $d > 4$.  

Given that $p < \frac{1}{2}d$, elementary algebra shows that $q> 2d - 1 - \frac{4}{d-2}$.  Since $d>6$ and $q$ is an integer,     $q \ge  2d-1$ as desired.
 \end{proof}

For bounds in the reverse direction, we have the following lemma.  
\begin{lemma}  If $d>8$, then $p > \frac{d-2}{3}$ and $q \le 3d + 16 $
\end{lemma}

\begin{proof} Again applying Theorem~\ref{thm:main} with $j = 1$ and $\newbeta = 1$, we have $$R(d-1) \le 3.$$  That is, at most two positive multiples of $p$ are in the interval $[0,d-2]$.  Thus, $p> \frac{d-2}{3}$ as desired.  

Simple algebra now yields that $q < 3d+7 + \frac{30}{d-5} < 3d + 17$.
\end{proof}

\subsection{A Bogomolov--Miyaoka--Yau based inequality}\label{ss:BMY}
 
We begin with a summary of a result of Orevkov~\cite{Or}  which is based on the Bogomolov-Miyaoka-Yau inequality~\cite{Miyo}.

Associated to each singular point on a curve $C$ there is an {\it Orevkov $\ol{M}$--number}, defined in full generality in~\cite{Or}; we note here that in the case of singularities having link a $(p,q)$--torus knot, that is, having a single Puiseux pair $(p,q)$, the value is  $\ol M  = p+q-[q/p]-3$.

We have the following consequence of the Bogomolov--Miyaoka--Yau inequality~\cite{Miyo}; because the details are fairly technical, we delay presenting them until Section~\ref{bmysubsec}. 

\begin{theorem}\label{thm:codimbound}
If $C\subset \C P^2$ is a cuspidal curve of genus $g>0$ and degree $d$ with singular points $z_1\ldots,z_n$ and corresponding $\ol{M}$ numbers
  $\ol{M}_1,\ldots,\ol{M}_n$, then
\begin{equation}\label{eq:bmybound}
\sum_{i=1}^n \ol{M}_i\le 3d+4g-5.
\end{equation}
\end{theorem}

\begin{example}\label{ex:stupidexample}
Theorem~\ref{thm:main} does not prohibit the existence of a curve of degree $3p$ ($p=1,2,\ldots$) with genus $1$ and a singularity $(p;9p+1)$.   One can indeed
check that this case satisfies \eqref{eq:main} for all $j$. Nevertheless it does not satisfy \eqref{eq:bmybound} if $p\ge 11$. In fact,
we have $\ol{M}=p+9p+1-9-3=10p-11$ and the bound is $\ol{M}\le 3d-1=9p-1$. This is satisfied only when $p\le 10$. \end{example}

\subsection{The multiplicity bound}\label{sec:multbound}

We will now prove a multiplicity bound   similar to   one given by Orevkov in~\cite[Theorem A]{Or}.  We restrict to the case of interest, $g=1$ and one singular point, but with care the argument extends to arbitrary genus and multiple singularities.  In the case of a simple singularity of type $(p,q)$, the multiplicity is  the minimum of $p$ and $q$ which, since we assume throughout that $p<q$, is given by $p$. 

\begin{theorem}\label{thm:multbound}
Suppose that $C$ is a cuspidal curve of 
degree $d$, genus $g = 1$, and with one singular point of multiplicity $m$.   Then  $d < \alpha m+\beta$, where $\alpha=\frac12(3+\sqrt{5})$ and $\beta=\frac32+\frac{11}{10}\sqrt{5}$.  
\end{theorem}
\begin{proof}
By~\cite[Proposition 2.9]{BZ1} we have the Milnor number $\mu$ and $\overline{M}$ number satisfy $\mu \le m (\ol{M}-m+2)$ (this is immediate for a singularity of simple type). Therefore, by the genus formula
\begin{equation*}\label{eq:deltaformula}
(d-1)(d-2)-2g\le   m (\ol{M} -m +2).
\end{equation*}
 Using the assumption that $g=1$ and Theorem~\ref{thm:codimbound}, it follows that 
\begin{equation*}\label{eq:summj}
d^2 - 3d \le m(3d -1-m +2).
\end{equation*}
This can be rewritten as 
\begin{equation*}
d^2 -3(1+m)d +(m^2 -m) \le 0.
\end{equation*}
Viewing this as a quadratic polynomial in $d$ yields 

\begin{equation*} 
 2d \le 3(1+m) + \sqrt{ 9(1+m)^2 -4(m^2 -m)  }.
\end{equation*}
This simplifies to 

\begin{equation*}
 2d \le 3 + 3m + \sqrt{5}m \sqrt{   1 +  \frac{22}{5m}+  \frac{9}{5m^2}    },
\end{equation*}
which we can rewrite as 
$$
 2d \le 3+ (3+\sqrt{5})m +    \sqrt{5}m \left(\sqrt{   1 +   \frac{22}{5m} +  \frac{9}{5m^2}    }-1\right). $$

The proof is completed by showing that for $m\ge 2$, 
$$m \left(\sqrt{   1 +  \frac{22}{5m} +  \frac{9}{5m^2}    }-1\right) < \frac{11}{5}.$$
This is an elementary exercise in calculus, perhaps most easily solved for substituting $m = \frac{1}{x}$ to consider 
$$\frac{\sqrt{ 1 +\frac{22}{5}x + \frac{9}{5}x^2} -1}{x}  $$ on the interval $(0,\frac{1}{2}]$.
The first derivative of this function is easily seen to be negative, and L'H\^opital's rule determines the limit at 0 to be $\frac{11}{5}$.
\end{proof}

\subsection{Classification theorem}\label{ss:classify}

Theorem~\ref{thm:class} will be deduced from the multiplicity bound (Theorem~\ref{thm:multbound}) along with a technical result, Lemma~\ref{lem:maintechnical}, which follows the proof of a sequence of  simpler lemmas.  Throughout the rest of this section we assume $C$ is a curve of degree $d$ and genus one, with exactly one singular point, and that singular point is of type $(p,q)$.   We remind the reader that $p <q$.

We need to introduce some notation. 
Let $\phi_0,\phi_1,\ldots$ be the sequence of Fibonacci numbers such that $\phi_0=0$, $\phi_1=1$ and $\phi_{n+1}=\phi_n+\phi_{n-1}$.  Most elementary texts on number theory include the necessary background, for instance regarding such facts as $\gcd(\phi_n, \phi_{n+1}) =1=  \gcd (\phi_n, \phi_{n+2}) $ as well as nonlinear relations, such as Cassini's Identity $\phi_{n-1}\phi_{n+1} - \phi_n^2  = (-1)^n$, and its generalization $\phi_{n-r}\phi_{n+r}-\phi_n^2=(-1)^{n-r+1}\phi_r^2$.

Our next step is to rule out some special cases of possible values of $p$.

\begin{lemma}\label{lem:numtheo1}
Suppose $C$ is as in the assumptions of Theorem~\ref{thm:class} and is not one of the exceptional cases $(p,q) = (2,5), (2,11), (3,10), $ or $(6,37)$.  Then  $p\ne \frac{\phi_{2j-1}}{\phi_{2j+1}}d$ for all  $j>0$.
\end{lemma}
\begin{proof}
Suppose that   $ p = \frac{\phi_{2j-1}}{\phi_{2j+1}}d$ for some $j$. Since $\phi_{2j-1}$ is coprime to $\phi_{2j+1}$, we see that $\frac{d}{\phi_{2j+1}}$ is an integer.

Since $(p-1)(q-1) = d(d-3)$, we have
\[(\phi_{2j-1}d-\phi_{2j+1})(q-1)=\phi_{2j+1}d(d-3).\]
The left hand side can be rewritten using the identity $\phi_{2j+1}=3\phi_{2j-1}-\phi_{2j-3}$  to give $$[(d-3)\phi_{2j-1}+\phi_{2j-3}](q-1) = \phi_{2j+1}d(d-3).$$
Taking these equalities modulo $d$ and $d-3$, respectively,  we arrive at 
$$\phi_{2j+1}(q-1)=0\bmod d$$ and 
$$\phi_{2j-3}(q-1)=0\bmod (d-3) .$$
Thus   $d$ divides $\phi_{2j+1}(q-1)$, $ d-3 $ divides  $\phi_{2j-3}(q-1)$ and  so $\lcm(d, d-3) $ divides $\lcm( \phi_{2j+1},\phi_{2j-3})(q-1)$.  The value of $\lcm(d,d-3)$ is either $d(d-3)$ or $d(d-3)/3$, depending upon whether or not $d$ is divisible by $3$.  In either case, we have $$d(d-3)|3\phi_{2j-3}\phi_{2j+1}(q-1).$$
Since $(p-1)(q-1) = d(d-3)$, if follows that 
 
\begin{equation}\label{eq:ddivides} (p-1)|3\phi_{2j-3}\phi_{2j+1}.\end{equation}
Notice that in the case that $d$ is not divisible by $3$, we   have the stronger constraint $(p-1)|\phi_{2j-3}\phi_{2j+1}$.

Denote $x=\phi_{2j-1}$ and $y=\frac{d}{\phi_{2j+1}}\in\Z$, so that $xy=p$. Notice that $x^2+1=\phi_{2j-1}^2+1=\phi_{2j-3}\phi_{2j+1}$, which follows from the basic identities satisfied by the Fibonacci numbers. By \eqref{eq:ddivides} $xy-1$ divides $3(x^2+1)$; that is, there exist $c>0$ such that
\[c(xy-1)=3x^2+3.\]
Taking both sides modulo $x$, we infer that $c=kx-3$ for some integer $k>0$. Substituting this, after simplifications we obtain
\begin{equation}\label{eq:numeric}
(ky-3)x=3y+k.
\end{equation}
This equation has only a finite number of positive integral solutions, which we now enumerate.  First, if $x = 1$, then $$y = 1 + \frac{6}{k-3}$$ and the only solutions for the triple $(x,y,k)$ are $\{ (1, 2, 9),  (1,3,6), (1,4,5), (1,7,4)\}$.  Similarly, the only solutions with $y=1$ are $\{ (2,1, 9), (3,1,6), (4,1,5), (7,1,4)\}$.  If $x\ge   2$ and $y\ge 2$, we write $$k = 3\left(\frac{x+y}{xy -1}\right).$$ An easy calculus exercise shows that on the domain $\{x\ge 2, y\ge2\}$ the maximum of the right hand side
is achieved at $(2,2)$, with value $k=4$.  For $k = 1, 2, 3, $ and $4$, one finds the only solutions for $(x,y, k)$ are $$\{ (4,13,1), (5,8,1),(8,5,1),(13,4,1), (2, 8,2),(8,2,2), \hskip.5in $$   $$ \hskip.5in(2,3,3),(3,2,3),(1,7,4), (2,2,4),(7,1,4)\}.$$ Thus, the values of $(x,y)$ to consider are $$\{(1,2), (1,3),(1,4), (1,7), (2,2), (2,3), (2,8), (4,13),(5,8)\} $$ and their symmetric pairs.

Recall that we have $x = \phi_{2j-1}$,  $y = \frac{d}{\phi_{2j+1}}$, and $p = xy$.  The only possibilities for $x$ are $x = 1, 2, 5, $ and $13$, in which case $y = \frac{d}{2} , \frac{d}{5}, \frac{d}{13}, $ and $\frac{d}{34}$, respectively.  The possible pairs $(x,y)$ are thus $$\{(1,2), (1,3), (1,4), (1,7),  (2,1), (2,2), (2,3), (2,8), (5,8), (13,4)\}.$$
An immediate calculation yields the following possibilities for $(p,d)$:
$$\{ (2,4), (3,6),(4,8), (7,14),(2,5),(4,10), (6, 15), (16,40),(40,104),(52,136)\}.$$  For most of these, the corresponding value of $q=1+\frac{d(d-3)}{p-1}$ is not  an integer.  The values of $(p,q,d)$ that can arise as {\it integer} triples are $$\{ (2,5,4),(3,10,6),(2,11,5),(6,37,15)\}.$$
\end{proof}

\begin{lemma}\label{lem:maintechnical}
Suppose $C$ is as in the assumptions of Theorem~\ref{thm:class} and is not one of the exceptional cases.  If $d\ge\phi_{2k-1}+2$ and $d\ge 6$, then $p<\frac{\phi_{2k-1}}{\phi_{2k+1}}d$.
\end{lemma}
\begin{proof}[Proof of Lemma~\ref{lem:maintechnical}]
We proceed by induction over $k$. The base case, $k=1$,  is that $p<\frac{1}{2}d$, which is the statement of  Lemma~\ref{lemma-d/2}.

Notice that the sequence
$$
a_k=\frac{\phi_{2k-1}}{\phi_{2k+1}}
$$
is decreasing,  converging to $\frac{2}{3+\sqrt{5}}$. So
suppose we have already proved that $p<\frac{\phi_{2k-3}}{\phi_{2k-1}}d$ for some $k$ and assume that $p>\frac{\phi_{2k-1}}{\phi_{2k+1}}d$ (by Lemma~\ref{lem:numtheo1}
we do not have to consider the possibility that  $p=\frac{\phi_{2k-1}}{\phi_{2k+1}}d$).

Assume momentarily that
$$q>\frac{\phi_{2k-1}}{\phi_{2k-3}}d.$$  Then  the number of the elements in the semigroup generated by $p$ and $q$ in the interval $[0,\phi_{2k-1}d]$ is the number of lattice points in the triangle
$$\{(x,y)\in\R^2_{\ge 0}\colon x\frac{p}{\phi_{2k-1}d}+y\frac{q}{\phi_{2k-1}d}\le 1\}.$$
This is at most the number of the lattice points in the triangle
$$T= \{(x,y)\in\R^2_{\ge 0}\colon x\frac{1}{\phi_{2k+1}}+y\frac{1}{\phi_{2k-3}} < 1\};$$  notice that since $p>\frac{\phi_{2k-1}}{\phi_{2k+1}}d$ and $q>\frac{\phi_{2k-1}}{\phi_{2k-3}}d$, 
we replaced the inequality  $\le 1$ with the strict inequality $<1$, essentially deleting the hypothenuse of the triangle.

Counting lattice points in a polygon with lattice points as vertices can be done using Pick's theorem. In our situation, though, the triangle is especially
simple so we can use an elementary argument, which can be found for instance in~\cite[page 64]{Ro}, to conclude that
the number of lattice points in $T$ is  

$$R= \frac{(\phi_{2k+1} +1)(\phi_{2k-3} +1)}{2} -1.$$  
Finally,   elementary properties of Fibonacci numbers permit us to rewrite this as
\[ R = \frac{(\phi_{2k-1}+1)(\phi_{2k-1}+2)}{2}-1.\]

To summarize, under  the assumptions that  $p>\frac{\phi_{2k-1}}{\phi_{2k+1}}d$ and $q>\frac{\phi_{2k-1}}{\phi_{2k-3}}d,$ we have that the number of elements in the semigroup generated by $p$ and $q$ in the interval $[0,  \phi_{2k-1}d]$ is at most $R$.  However,  Theorem~\ref{thm:main} with $j = \phi_{2k-1}$ and $\newbeta =0$ states that $$0 \le R(\phi_{2k-1}d+1) -   \frac{(\phi_{2k-1}+1)(\phi_{2k-1}+2)}{2} \le 1,$$ and in particular,

$$R(\phi_{2k-1}d+1) \ge  \frac{(\phi_{2k-1}+1)(\phi_{2k-1}+2)}{2} .$$  Noting that $R = R(\phi_{2k-1}d+1)$ yields a contradiction.  (We have used here that $d\ge \phi_{2k-1} +2$, since Theorem~\ref{thm:main} requires that $j\le d-2$.)  Thus,  \eqref{eq:main} is not satisfied for $j=\phi_{2k-1}$ and $b=0$.

With this contradiction, we can now conclude that under the assumption   $p>\frac{\phi_{2k-1}}{\phi_{2k+1}}d$ we must have  
$q\le\frac{\phi_{2k-1}}{\phi_{2k-3}}d.$

Recall that our induction hypothesis is that  $p<\frac{\phi_{2k-3}}{\phi_{2k-1}}d$. That $p$ is an integer implies that 
\[p\le \frac{\phi_{2k-3}}{\phi_{2k-1}}d-\frac{1}{\phi_{2k-1}}.\]
We can use these inequalities to conclude 
\[(p-1)(q-1)\le \left(\frac{\phi_{2k-3}}{\phi_{2k-1}}d-\frac{1}{\phi_{2k-1}}-1\right)\left(\frac{\phi_{2k-1}}{\phi_{2k-3}}d-1\right).\]
This can be written as
\begin{equation}\label{eq:rhsofpq}
(p-1)(q-1)\le d^2-\left(\frac{\phi_{2k-3}}{\phi_{2k-1}}+\frac{\phi_{2k-1}}{\phi_{2k-3}}\right)d+1-\frac{d}{\phi_{2k-3}}+\frac{1}{\phi_{2k-1}}.
\end{equation}  
The term in parenthesis can be rewritten as
\[\frac{\phi_{2k-3}}{\phi_{2k-1}}+\frac{\phi_{2k-1}}{\phi_{2k-3}}=\frac{\phi_{2k-3}}{\phi_{2k-1}}+\frac{\phi_{2k-1}}{\phi_{2k-3}}+\frac{\phi_{2k+1}}{\phi_{2k-1}}-\frac{\phi_{2k+1}}{\phi_{2k-1}}.\]
Using the facts that $\phi_{2k-3} + \phi_{2k +1} =3\phi_{2k-1}$ and $\phi_{2k-3}\phi_{2k+1} - \phi_{2k-1}^2 = 1$, the first and the third term yield $3$, while the second and the last give $-\frac{1}{\phi_{2k-1}\phi_{2k-3}}$, so (\ref{eq:rhsofpq})  can be rewritten as 

\begin{equation}\label{eq:rhsofpq1}
  (p-1)(q-1) \le  d^2- \left( 3 - \frac{1}{ \phi_{2k-3}\phi_{2k-1}  } \right )d+ 1-\frac{d}{\phi_{2k-3}}+\frac{1}{\phi_{2k-1}},
 \end{equation} 
which can be rewritten as
\[ (p-1)(q-1) \le  d^2-3d+\frac{d}{\phi_{2k-1}\phi_{2k-3}}+1-\frac{d}{\phi_{2k-3}}+\frac{1}{\phi_{2k-1}}.\]

Since $(p-1)(q-1) = d^2 - 3d$, this implies that   
$$\frac{d}{\phi_{2k-1}\phi_{2k-3}}+1-\frac{d}{\phi_{2k-3}}+\frac{1}{\phi_{2k-1}}\ge  0 .$$  This is equivalent to  
$$d \le  \phi_{2k-3}\left( \frac{\phi_{2k-1} +1}{\phi_{2k-1} -1}\right).$$  
The conditions of the theorem include $d \ge \phi_{2k-1} +2$, so this equality implies
$$\phi_{2k-1} +2 \le \phi_{2k-3}\left( \frac{\phi_{2k-1} +1}{\phi_{2k-1} -1}\right).$$
Clearing the denominator and writing $\phi_{2k-3} = \phi_{2k-1} - \phi_{2k-2}$ yields
$$(\phi_{2k-1} -1)(\phi_{2k-1} +2)   \le  (\phi_{2k-1} - \phi_{2k-2})(\phi_{2k-1} +1).$$  Expanding, this becomes
$$\phi_{2k-1}^2 + \phi_{2k-1} -2 \le \phi_{2k-1}^2 +\phi_{2k-1} - \phi_{2k-2}\phi_{2k-1} - \phi_{2k-2}.$$  Finally, this can be rewritten as
$$  \phi_{2k-2}\phi_{2k-1} + \phi_{2k-2}\le 2,$$ which is false for $k\ge 1$.

With this final contradiction  we see that  $p\le\frac{\phi_{2k-1}}{\phi_{2k+1}}d$ must hold, so the induction step is accomplished.
\end{proof}

We shall need another result.

\begin{lemma}\label{lem:pplusq}
If $d>6$ and  $p+q>3d$, then $p\le \frac{3}{8}{d}$; equivalently, if $p>\frac38d$ then $p+q\le 3d$.
\end{lemma}
\begin{proof}
Suppose that  $p+q>3d$ and $8p>3d$.  The eight multiples $ap$ with $a= 0, \ldots, 7$, are possibly in $[0,3d]$, but $8p$ is not. The conditions imply that  $p+aq$ is not in the interval for any $a>0$. It is possible that $q$ is in the interval, but $2q$ is not, since in Lemma~\ref{lemma-d/2} we showed  that $q \ge 2d-1$.  This gives a maximum of   $9$ elements in $S\cap[0,3d]$, while \eqref{eq:main} for $j=3$ and $b=0$ implies that there
must be at least $10$.
\end{proof}

\begin{proof}[Proof of Theorem~\ref{thm:class}]
We suppose $d>6$;  for $d\le 6$ the result is a straightforward computation.

We consider three cases.  The first  case is  $p\le\frac{3}{8}d$. Combining this with the result of  Theorem~\ref{thm:multbound}, which states that
$d<\alpha p+\beta$, yields $(\frac83-\alpha )p<\beta$. This places an upper bound on $d$; performing the arithmetic and using simple bounds on $\alpha$ and $\beta$ yields $d\le 300$.   All of these can be analyzed with a computer search, which yields the exceptional cases  (a)--(h) of Theorem~\ref{thm:multbound}.  Notice  that the only
examples having degree more than $33$ are in item (h) of that list, but the BMY inequality rules these out:  see Example~\ref{ex:stupidexample} following Theorem~\ref{thm:codimbound}.

Suppose now that  $p>\frac{3}{8}d$. By Lemma~\ref{lem:pplusq},  $p+q\le 3d$. The second case is that $p+q\le 3d-1$, so $q\le 3d-1-p$. Substituting this into
\eqref{eq:main} we obtain
\[(p-1)(3d-p-2)\ge d(d-3).\]
We proceed as in the proof of Theorem~\ref{thm:multbound}. The inequality can be rewritten in terms of a quadratic polynomial in $d$:
$$d^2 -3p d +(p^2+p-2) \le 0.$$
Applying the quadratic formula yields
$$2d \le 3p+\sqrt{5p^2 - 4p +8}  = 3p +  \sqrt{5}  p   \sqrt{    1 -\frac{4}{5p} + \frac{8}{5p^2}   }.$$This in turn can be written
$$2d \le  3p +  \sqrt{5}  p + \sqrt{5}p \left(   \sqrt{    1 -\frac{4}{5p} + \frac{8}{5p^2}   } -1\right).$$The term in parenthesis equals 0 for $p=2$ and is negative for $p>2$.  Thus, in general, we have   $d < \alpha p$.

For some $k$,  $d\in[\phi_{2k-1}+2,\phi_{2k+1}+1]$. By Lemma~\ref{lem:maintechnical} we have  $p< \frac{\phi_{2k-1}}{\phi_{2k+1}}d$.  Combined with $d < \alpha p$ we find 
\begin{equation}\label{boundineq}
\frac{\phi_{2k+1}}{\phi_{2k-1}} < \frac{d}{p} < \alpha.
\end{equation}

The sequence $\frac{\phi_{2k+1}}{\phi_{2k-1}}$ are the even convergents of the continued fraction expansion of $\alpha$.  As such, they form an   increasing sequence converging to $\alpha$, offering the closest lower approximations for   given denominators.  More precisely,  ~(\ref{boundineq}) implies that $p> \phi_{2k-1}$.   (See, for instance,~\cite{niven-zuckerman}, for these results concerning continued fractions. In particular, Theorem 7.13 of~\cite{niven-zuckerman} states the required result concerning the sense in which convergents of continued fractions provide the best rational approximations  to an irrational number.)

We now have $\phi_{2k-1} <  p< \frac{\phi_{2k-1}}{\phi_{2k+1}}d$, implying that $d> \phi_{2k+1}$, and thus $d= \phi_{2k+1}+1$.  We now have
$$\frac{\phi_{2k+1}}{\phi_{2k-1}} < \frac{\phi_{2k+1}+1}{p},$$
so $$p< \phi_{2k-1} + \frac{ \phi_{2k-1} }{\phi_{2k+1}}.$$  Since $p$ is an integer,
$$p \le \phi_{2k-1}, $$ a contradiction.

The last case is $p+q=3d$. Expanding  $(p-1)(q-1)=d(d-3)$, we find $p+q = pq+ 3d -d^2 +1$.  Combining these gives  that $pq=d^2-1$.  Writing $(q-p)^2 = (q+p)^2 -4pq$ yields $(q-p)^2=5d^2+4$. 

We need to find all the integers $d$ such that $5d^2+4$
is a square. This problem is a case of  Pell's equation.  One   accessible reference is~\cite[Section 5.1]{FLMN04}, where this precise problem is solved.  
The result states 
that $5d^2+4$ is a square if and only if $d=\phi_{2i}$ for some integer $i$.    For  a more general number theoretic discussion, a good reference is the discussion  of the {\it Unit Theorem} in\cite{marcus}.

As a brief aside, we include here a summary of the argument. 

\begin{lemma}  If  $(x,y)$ is an integer solution to $5y^2 +4 = x^2$ then $y = \pm \phi_{2n}$ for some $n$.
\end{lemma}

 \begin{proof}   This equation  can be rewritten as $(\frac{x}{2} + \frac{y}{2}\sqrt{5}) (\frac{x}{2} - \frac{y}{2}\sqrt{5}) =1$.  Notice that $x$ and $y$ must have the same parity, and thus solutions correspond precisely to units of norm one   in the algebraic number ring $\Z[\frac{1 + \sqrt{5}}{2}]$.  (Units have norm either plus or minus one.)  For a real quadratic number ring, the set of units forms an abelian group isomorphic to $\Z_2 \oplus \Z$ (see, for instance,~\cite{marcus} or~\cite{cohn}). In our case, the infinite summand is generated by an element of the form $\frac{a}{2} + \frac{b}{2}\sqrt{5}$, where $a$ and $b$ are positive and have the same parity.  Clearly, a generator of this form will have the minimum possible value of $a$.  Since $\gamma = \frac{1}{2} + \frac{1}{2}\sqrt{5}$ does have norm $-1$, this is the generator of the set of units modulo torsion. 

The first five  powers of $\gamma$ are $ \frac{1}{2} + \frac{1}{2}\sqrt{5}$,  $\frac{3}{2} + \frac{1}{2}\sqrt{5}$, $ \frac{4}{2} + \frac{2}{2}\sqrt{5}$,    $\frac{7}{2} + \frac{3}{2}\sqrt{5}$ and   $\frac{11}{2} + \frac{5}{2}\sqrt{5}$.  Notice the numerators of the coefficients of $\sqrt{5}$ in $\gamma^n$ are the  Fibonacci numbers, $\phi_n$ and the numerator of the rational parts can be expressed as $3\phi_{n-1} + \phi_{n-2}$.  For instance, $11 = 3 \cdot 3 +2$.  That this pattern continues is an easy inductive argument using the defining recursion relation for  Fibonacci numbers.   Finally, since $\gamma$ has norm $-1$, 
  only even powers of $\gamma$ have norm one, and thus only the Fibonacci numbers $\phi_{2n}$ appear as solutions for $y$ in our original equation $5y^2 +4 = x^2$.
 \end{proof}

Solving the pair of equations $p+q=3d$ and  $(p-1)(q-1)=d(d-3)$ for $p$ and $q$, with $q>p$, yields
\begin{align*}
p&=\frac32d-\frac12\sqrt{5d^2+4}=\phi_{2i-2}\\
q&=\frac32d+\frac12\sqrt{5d^2+4}=\phi_{2i+2}.
\end{align*}
Notice, that $\gcd(p,q)=\gcd ( \phi_{2i-2}, \phi_{2i+2}   ) = \phi_{\gcd(2i-2,2i+2)}$. If $i$ is odd, then $p$ and $q$ are both divisible by $3$, so they are not coprime, and the
case is ruled out. We are left with the case $p=\phi_{4i-2}$, $q=\phi_{4i+2}$ and $d=\phi_{4i}$.
\end{proof}

\subsection{Construction of curves}
We will now use Orevkov's argument (see~\cite[Section 6]{Or}) to construct curves with $(p,q)=(\phi_{4j-2},\phi_{4j+2})$ and degree $\phi_{4j}$.

\begin{proposition}\label{prop:orev}
For any $j=1,2,\ldots$ there exists a curve of genus $1$ and degree $\phi_{4j}$ having a unique singularity with of type $(\phi_{4j-2},\phi_{4j+2})$.
\end{proposition}
\begin{proof}
Fix a curve $N$ of degree $3$ with one node. Let $f\colon\C P^2\to\C P^2$ be a Cremona transformation as in~\cite{Or}. Let $F_1$ be a cubic
that passes through the node of $N$ and is tangent to one branch with the tangency order eight. Such curve exists by a parameter counting argument; that
is, the space of all cubics has dimension $\binom{3+2}{2}-1=9$, we need one parameter to make $F_1$ pass through the node, and each order of
tangency is one more condition, so we need altogether eight conditions. Notice that $F_1$ has genus one and does not intersect $N$ away from the node.

We define inductively $F_j=f(F_{j-1})$. Since $f$ is biregular away from $\C P^2\setminus N$, each curves $F_j$ has genus one and a single cusp.
The characteristic sequence of the point of $F_1$  which is tangent to $N$ is $(1,8)$ (by this we mean that it is a smooth point of $F_1$ and the order
of tangency is eight), and $(1,8)=(\phi_2,\phi_6)$. The image of this point under the composite $f\circ f
\circ\ldots\circ f$ is the singular point of $F_j$ and the characteristic sequence becomes $(\phi_{4j-2},\phi_{4j+2})$ by the same
argument as in~\cite{Or}. The degree of $F_j$ is computed via the genus formula and the relation $(\phi_{4j-2}-1)(\phi_{4j+2}-1)=\phi_{4j}(\phi_{4j}-3)$.%
\end{proof}

The following facts are   consequences  of  explicit constructions, which will not be given here. For the first two cases, explicit formula for the curves can be given;  
we are thankful to Karoline Moe for describing a construction of the last curve to us.
\begin{proposition}\label{prop:moe}
Cases (a), (b) and (c) from Theorem~\ref{thm:class} can be realized; that is, there exists a curve of degree $4$ with singularity $(2;5)$, a curve
of degree $5$ with singularity $(2;11)$, and a curve of degree $6$ with singularity $(3;10)$.
\end{proposition}

The following result is well known to experts, we refer to~\cite{artal2} for a modern approach.
\begin{proposition}\label{prop:artal}
There exists a curve of degree $6$ with a singularity $(2;19)$. This  is case (g) from the list with $p=2$.
\end{proposition}

\subsection{The BMY inequality}\label{bmysubsec}
Here we provide background for the proof of  Theorem~\ref{thm:codimbound}. 
Our  approach
closely follows~\cite{BZ3, Or}.  The Bogomolov--Miyaoka--Yau inequality, see~\cite{KNS,Miyo}, is one of the main tools in studying curves in algebraic surfaces.

To formulate the BMY inequality we need some preliminaries. We let $X$ be a (closed) algebraic surface. 
Recall that a \emph{divisor} on $X$ is a formal sum $\sum \alpha_i D_i$, where $\alpha_i\in\Z$ and $D_i$ are closed algebraic curves on $X$.
One of the main examples is  the \emph{canonical divisor} $K$. This is a divisor which represents a class in $H_2(X;\Z)$ that  is  Poincar\'e dual to the first Chern class of the cotangent bundle of $X$. 
  
Let $D$ be a 
reduced effective divisor (that is, each irreducible component of $D$, which is a reduced algebraic curve, has coefficient one)
with the property that $X\setminus D$
is of log--general type. We refer to~\cite[Section I.1]{Mn} for the definition of  log--general type and note that in our applications $X\setminus D$ will always be of this type.
There exists a so called Zariski--Fujita decomposition of the divisor $K+D$; this is a unique  decomposition
$K+D=H+N$, where $H$ and $N$ are rational divisors  and $H$ is the {\it numerically effective} (in~\cite{Mn} this is called  ``arithmetically effective'')  part and $N$ is 
the negative part of $K+D$; see~\cite{Fuj} or~\cite[Section I.3]{Mn}.

The two fundamental proprieties of this decomposition are that $H\cdot N=0$ and $N^2\le 0$. The BMY
inequality as given in~\cite{KNS} or~\cite[Theorem 2.1]{Or} says that $H^2\le 3\chi(X\setminus D)$, 
where $\chi$ is the Euler characteristic; for our purpose the following formulation is sufficient.

\begin{theorem}[BMY inequality]\label{thm:BMY}
Suppose  $X$ is an algebraic surface, $K$ its canonical divisor, and $D$ a divisor on $X$ such that $X\setminus D$ is of log--general type (see~\cite{Fuj}). Then
\begin{equation}\label{eq:BMY}
(K+D)^2\le 3\chi(X\setminus D).
\end{equation}
If, in addition,  in the Zariski--Fujita decomposition $K+D=H+N$ we have $N\neq 0$, 
then we cannot have an equality in \eqref{eq:BMY}.
\end{theorem}

Suppose that $C\subset \C P^2$ is a cuspidal curve of positive genus $g>0$ with singular points $z_1,\ldots,z_n$,  $n>0$. 
For some $m>0$, appropriately blowing up $m$ points resolves the singularities, providing what is called a {\it good resolution} (also known as an SNC resolution,
where SNC stays for ``simple normal crossings''); in particular, it constructs a curve $C'$, {\it the strict transform of $C$}, in a manifold $X$ diffeomorphic to 
$\C P^2 \#_m\overline{\C P}^2$.

The steps  of forming the good resolution of $C$ build  a sequence of divisors in $X$, $E_1,\ldots,E_m$, each of multiplicity one (they corresponds to the exceptional divisors
of the blow-ups constituting the good resolution).   The {\it reduced exceptional divisor} $E$ is the sum $\sum E_i$; see~\cite[Section 8.1]{Wall04}. We set
\[D=C'+E.\]  
This is a reduced effective divisor on $X$. 

 A result of Wakabayashi~\cite{Wa} states that the complement of a positive genus algebraic curve in $\C P^2$ of degree $d\ge 4$ is of log--general type.  By the genus formula, any curve of degree $3$ or less is either nonsingular or   genus 0.  In particular, $C$ is of degree four or more and Wakabayashi's result implies that the complement $\C P^2\setminus C$   is of log--general type~\cite{Wa}.  Since $\C P^2 \setminus C \cong X \setminus D$, we have  $X \setminus D$ is log--general type, so Theorem~\ref{thm:BMY} applies.

In order to show that the inequality in \eqref{eq:BMY} is sharp, we use the following result proved in~\cite{OZ}; see~\cite{OZ2} for more detailed exposition.

\begin{lemma}[see~\cite{OZ}]
If $C$ has $n$ cuspidal singular points and $K+D=H+N$ is the Zariski-Fujita decomposition, then $N^2<-\frac{n}{2}$. In particular, if $C$ has at least
one cuspidal singular point, then $N$ is not trivial.
\end{lemma}

Since $X \setminus D \cong \C P^2 \setminus C$, we have $\chi(X \setminus D) = (2g+1)$.   Thus,~\eqref{eq:BMY} becomes $$(K+D)^2 < 3(2g+1).$$  This can be written as
$$K(K+D) + D(K+D) < 6g +3.$$  
By the adjunction formula  $D(K+D) = 2g -2$; see~\cite[Section 7.6]{Wa}. Substituting this, we obtain
$$K(K+D) < 4g + 5.$$

The homology of $X$ splits as an orthogonal sum, with one summand spanned by $L$ (representing a generator of $H_2(\C P^2)$) and separate summands,  one for each singular point.  
Details are presented in~\cite[Section 2]{Or}.
Accordingly, we write $K=K_0+K_1+ \cdots+ K_n$ and $D=D_0+D_1+\cdots+D_n$. Here $K_0$ and $D_0$ belong to the summands spanned by $L$ and $K_i$,  and the  $D_i$
belong to the summands corresponding to the singular points $z_i$. Note that $K_0=-3L$ and $D_0=dL$.

Using this decomposition, we  can write the inequality as a summation: 
\[K_0(K_0 +D_0) + \sum_{i=1}^{ n} K_i(K_i +D_i) < 4g +5.\]
Substituting the values of $K_0$ and $D_0$ we obtain.
$$ 9 -3d + \sum_{i=1}^n K_i(K_i +D_i) < 4g +5.$$
According to~\cite[Proposition 4.1]{BZ3},  $K_i(K_i+D_i)$ can be identified with the Orevkov $\ol{M}$--number  (where it was called the codimension). Thus, 
$$\sum_{i>0} \ol{M}_i < 3d +4g -4.$$ 
As both sides of the above inequality are integers, we have
$$\sum_{i>0} \ol{M}_i \le 3d+4g -5.$$
Theorem~\ref{thm:codimbound} is proved.

\section{The   semicontinuity of the spectrum}\label{ss:semicont}

The spectrum $\Sigma$ of a singular point of a plane curve is a collection of rational numbers from the interval $(0,2)$, where each
rational number can occur multiple times. The count  with multiplicity,  $\#\Sigma$, is the Milnor number of the singularity. It is one of the strongest invariants of
singularities. From a topological point of view, the spectrum can be (almost) recovered from the Tristram--Levine signatures of the link.  
For a singularity $x^p-y^q=0$ (that is a singularity  whose link is $T(p,q)$), the spectrum is the set
\[\Sigma_{p,q}=\left\{\frac{i}{p}+\frac{j}{q},\ 1\le i\le p-1,\ 1\le j\le q-1\right\},\]
where if a number $x$ can be presented in $\nu$ different ways as a sum $\frac{i}{p}+\frac{j}{q}$, it means that $x$ appears in $\Sigma_{p,q}$ with multiplicity
$\nu$.

There is a property of semicontinuity of spectra. Following~\cite{FLMN06} we will formulate it as follows.

Suppose $C$ is an algebraic curve in $\C P^2$  of arbitrary genus and not necessarily cuspidal. Suppose $\deg C=d$. Let $z_1,\ldots,z_n$ be
the singular points and $\Sigma_1,\ldots,\Sigma_n$ the corresponding \emph{spectra}. Let
\[\Sigma_{d,d}:=\left\{\frac{i}{d}+\frac{j}{d},\ 1\le i,j\le d-1\right\}\]
be the spectrum of the singularity $x^d-y^d=0$. Then for any $x\in\R$ we have
\begin{equation}\label{eq:semics}
\#(\Sigma_{d,d}\cap(x,x+1))\ge\sum_{j=1}^n\#(\Sigma_j\cap(x,x+1)).
\end{equation}

Equation~\eqref{eq:semics}, the spectrum semicontinuity property, 
is one of the strongest obstructions to the existence of curves in $\CP^2$ with prescribed singularities. It is most effective if the total number of elements of
the spectra $\sum\#\Sigma_j$ is close to $\#\Sigma_{d,d}=(d-1)^2$, that is, if the (geometric) genus of $C$ is small. It was effectively used
in~\cite{FLMN06} to classify rational cuspidal curves with one cusp and one Puiseux pair at that cusp. We will show  is of limited effectiveness in case of curves of genus one.

Substituting $x=-1+\frac{l}{d}$. into \eqref{eq:semics}, where $l=1,\ldots,d-1$ we obtain
\begin{equation}\label{eq:ssd}
\sum_{j=1}^n\#\left(\Sigma_j\cap(0,\frac{l}{d})\right)\le \frac12(l-1)(l-2).
\end{equation}
This equation in~\cite{FLMN06} is referred to as $(SS_l)$. We shall examine how these inequalities apply to the classification problem of cuspidal curves of genus one
with one singular point and one Puiseux pair, as in Theorem~\ref{thm:class}.

\begin{example}\label{ex:t47-spec}
The case where $C$ has degree $6$ and is of simple type $(4,7)$ (so its genus is $1$) satisfies all the $SS_l$ inequalities, but cannot occur
by discussion in Section~\ref{examplessect}. In fact there is only one singular point with spectrum
\[\Sigma=\left\{\frac{11}{28},\frac{15}{28},\frac{18}{28},\frac{19}{28},\frac{23}{28},\frac{25}{28},\frac{26}{28},\frac{27}{28},\ldots\right\}\]
(the spectrum is symmetric around $1$, so we give only elements in spectrum in the interval $[0,1]$.
The values of $\#\Sigma\cap(0,\frac{l}{d})$ for $l=1,\ldots,6$ are $0,0,1,3,6,9$, which are less than or equal to  $0,0,1,3,6,10$, as given by the
right hand side of \eqref{eq:ssd}. Theorem~\ref{thm:main}, however, obstructs the existence of such curve, see Section~\ref{ss:t47}.
\end{example}

Similarly, one can show that the property $(SS_l)$ admits, for example, a genus one curve of degree $75$ and with a singularity of type $(28,201)$.

\begin{example} If $C$ is a curve of simple type $(p,q)$, then, according to~\cite[Example 2.4]{FLMN06}, $(SS_{d-1})$ reads as
\[\frac{(p-1)(q-1)}{2}+\left\lfloor\frac{q}{d}\right\rfloor-\left\lfloor\frac{pq}{d}\right\rfloor\le\frac{(d-2)(d-3)}{2}.\]
Since $(p-1)(q-1)=d(d-3)$, this gives
\[\left\lfloor\frac{pq}{d}\right\rfloor\ge d-3+\left\lfloor\frac{q}{d}\right\rfloor.\]
Writing $pq=d(d-3)+p+q-1$ we arrive at
\[\left\lfloor\frac{p+q-1}{d}\right\rfloor\ge \left\lfloor\frac{q}{d}\right\rfloor.\]
This inequality is trivially satisfied whenever $p\ge 1$.
\end{example}

\end{document}